\documentclass[12pt]{article}
\pdfoutput=0
\usepackage{amsthm}
\usepackage[margin=1in]{geometry}
\usepackage[pdftitle={Dyck tilings, increasing trees, descents, and inversions},
            pdfauthor={Jang Soo Kim, Karola Meszaros, Greta Panova, and David B. Wilson},
            bookmarks=true,bookmarksopen=true,bookmarksopenlevel=2]{hyperref}
\usepackage{amsmath,paralist}
\usepackage{amssymb}
\usepackage{mathtools}
\usepackage{graphicx}
\usepackage{subfigure}
\usepackage{mathrsfs}
\newcommand{\tz}[2]{#2}
\tz{}{
\usepackage{tikz,color}
\usetikzlibrary{arrows,shapes,decorations,backgrounds,chains,external}
}

\def\clap#1{\hbox to 0pt{\hss#1\hss}}

\def\mathclap{\mathpalette\mathclapinternal}

\def\mathclapinternal#1#2{\clap{$\mathsurround=0pt#1{#2}$}}

\newcommand{\sref}[1]{\hyperref[#1]{\S~\ref*{#1}}}
\newcommand{\aref}[1]{\hyperref[#1]{Appendix~\ref*{#1}}}
\newcommand{\lref}[1]{\hyperref[#1]{Lemma~\ref*{#1}}}
\newcommand{\tref}[1]{\hyperref[#1]{Theorem~\ref*{#1}}}
\newcommand{\cref}[1]{\hyperref[#1]{Corollary~\ref*{#1}}}
\newcommand{\fref}[1]{\hyperref[#1]{Figure~\ref*{#1}}}
\newcommand{\frefs}[2]{\hyperref[#1]{Figure~\ref*{#1}#2}}
\newcommand{\pref}[1]{\hyperref[#1]{Proposition~\ref*{#1}}}
\newcommand{\tblref}[1]{\hyperref[#1]{Table~\ref*{#1}}}

\newcommand{\MRhref}[2]{\href{http://www.ams.org/mathscinet-getitem?mr=#1}{MR#2}}
\makeatletter
\def\@strippedMR{}
\def\@scanforMR#1#2#3\endscan{%
  \ifx#1M\ifx#2R\def\@strippedMR{#3}%
  \else\def\@strippedMR{#1#2#3}%
  \fi\fi}
\def\@rst #1 #2other{#1}
\newcommand\MR[1]{\relax\ifhmode\unskip\spacefactor3000 \space\fi
  \MRhref{\expandafter\@rst #1 other}{#1}}
\newcommand\MRs[1]{\relax\ifhmode\unskip\spacefactor3000 \space\fi
  \@scanforMR#1\endscan
  \MRhref{\@strippedMR}{\@strippedMR}}
\makeatother

\newcommand{\old}[1]{}
\newtheorem{theorem}{Theorem}[section]
\newtheorem{proposition}[theorem]{Proposition}
\newtheorem{lemma}[theorem]{Lemma}

\newtheorem{conjecture}[theorem]{Conjecture}

\renewcommand{\th}{${}^{\text{th}}$ }

\newcommand{\U}{\mathrm{U}}
\newcommand{\rD}{\mathrm{D}}
\newcommand{\op}{\texttt{(}}
\newcommand{\cp}{\texttt{)}}
\newcommand{\eq}{\begin{equation}}

\newcommand{\art}{\operatorname{art}}
\newcommand{\des}{\operatorname{des}}
\newcommand{\dis}{\operatorname{dis}}
\newcommand{\sh}{\operatorname{sh}}
\newcommand{\DTr}{\operatorname{DTR}}
\newcommand{\DTs}{\operatorname{DTS}}
\newcommand{\match}{\operatorname{match}}
\newcommand{\matchmin}{\operatorname{min-word}}
\newcommand\D{\mathcal{D}}

\newcommand\zigzag{\mathrm{zigzag}}

\newcommand\qint[1]{\left[ #1\right]_q}
\newcommand\LL{\mathscr{L}}

\newcommand\tiles{\operatorname{tiles}}
\newcommand\area{\operatorname{area}}
\newcommand\sgrow{\operatorname{strip-grow}}
\newcommand\rgrow{\operatorname{ribbon-grow}}
\newcommand\mad{\operatorname{mad}}

\newcommand\inv{\operatorname{inv}}
\newcommand\INV{\operatorname{INV}}
\newcommand\desdif{\operatorname{desdif}}
\newcommand\res{\operatorname{res}}
\newcommand\REM{\operatorname{REM}}
\newcommand\DES{\operatorname{DES}}

\def\P{P_\lambda}
\def\L{\mathscr{L}(\P)}

\def\Z{\mathbb{Z}}

\begin{document}

\tz{}{
% nodes with cdots
\tikzstyle{w}=[label=right:$\textcolor{red}{\cdots}$]
\tikzstyle{b}=[label=right:$\cdot\,\textcolor{red}{\cdot}\,\cdot$]
\tikzstyle{bb}=[circle,draw=black!90,fill=black!100,thick,inner sep=1pt,minimum width=3pt]
\tikzstyle{bb2}=[circle,draw=black!90,fill=black!100,thick,inner sep=1pt,minimum width=2pt]
\tikzstyle{b2}=[label=right:$\cdots$]
\tikzstyle{w2}=[]
\tikzstyle{vw}=[label=above:$\textcolor{red}{\vdots}$]
\tikzstyle{vb}=[label=above:$\vdots$]

% Set the overall layout of the tree
\tikzstyle{level 1}=[level distance=3.5cm, sibling distance=3.5cm]
\tikzstyle{level 2}=[level distance=3.5cm, sibling distance=2cm]

% Define styles for bags and leafs
\tikzstyle{bag} = [text width=4em, text centered]
\tikzstyle{end} = [circle, minimum width=3pt,fill, inner sep=0pt]

%Macros for drawing Dyck tiling, tableaux, etc

\newcommand\DyckPath[2][]{
  \def\xCoord{0} \def\yCoord{0}
  \def\Path{[#1](0,0)}
  \foreach \num [count=\n from 1] in {#2}
   {
    \ifodd\n
     \global\edef\xCoord{\xCoord + \num}
    \else
     \global\edef\yCoord{\yCoord + \num}
    \fi
    \global\edef\Path{\Path -- (\xCoord,\yCoord)}
   }
   \expandafter\path\Path;
 }
 \newcommand\DyckDot[2][0.125]{
  \fill [shift = {(0.5, -0.5)}] (#2) circle [radius = #1];
 }
 \newcommand\DyckShade[2][gray]{
  \draw[fill = #1] (#2) rectangle +(1,-1);
 }
 \newcommand\DyckDiag[2]{
  \begin{scope}[shift = {(#2)}]
  \foreach \n in {1,...,#1} {\DyckShade{\n - 1, \n-1}}
  \end{scope}
 }

  \newcommand\DyckTile[2][]{
  \def\xCoord{0} \def\yCoord{0}
  \def\Path{[#1](0,0)}
  \foreach \num [count=\n from 1] in {#2}
   {
    \ifodd\n
     \global\edef\xCoord{\xCoord + \num}
    \else
     \global\edef\yCoord{\yCoord + \num}
    \fi
    \global\edef\Path{\Path -- (\xCoord,\yCoord)}
   }
   \expandafter\path\Path;
 }

 \newcommand\Ribbon[2][]{
  \def\xCoord{0} \def\yCoord{0}
  \def\Path{(0,0)}
  \def\RPath{(0,0)}
  \def\LastPoint{(0,-1)}
  \foreach \num [count=\n from 1] in {#2}
   {
    \global\edef\RPath{  \LastPoint -- \RPath }
    \global\edef\LastN{\n}
    \ifodd\n
     \global\edef\xCoord{\xCoord + \num}
    \else
     \global\edef\yCoord{\yCoord + \num}
    \fi
    \global\edef\Path{\Path -- (\xCoord,\yCoord)}
    \global\edef\LastPoint{(\xCoord-1,\yCoord-1)}
   }
   \edef\RPath{ (\xCoord-1,\yCoord) -- \RPath}
  % \tracingonline=1\show\RPath
   \edef\Path{[#1] \Path -- \RPath }
   \expandafter\draw\Path;
 }

\tikzset{every picture/.style ={baseline ={(0,0)} } }
\tikzset{vertex/.style = {fill, circle, inner sep = 1pt}}
\usetikzlibrary{decorations.pathreplacing}
}

\title{\vspace*{-18pt}Dyck tilings, increasing trees, descents, and inversions}
\date{}
\author{
\begin{tabular}{c}
\href{http://newton.kias.re.kr/~kimjs/}{Jang Soo Kim}\thanks{Current affiliation is Korea Institute for Advanced Study.}\\[-3pt]
\small University of Minnesota
\end{tabular} \and
\begin{tabular}{c}
\href{http://www.math.cornell.edu/~karola/}{Karola M\'esz\'aros}\thanks{Supported in part by a National Science Foundation Postdoctoral Fellowship (DMS 1103933).} \thanks{Current affiliation is Cornell University.}\\[-3pt]
\small University of Michigan
\end{tabular} \and
\begin{tabular}{c}
\href{http://www.math.ucla.edu/~panova/}{Greta Panova}\thanks{Supported by a Simons Postdoctoral Fellowship.}\\[-3pt]
\small University of California at Los Angeles
\end{tabular} \and
\begin{tabular}{c}
\href{http://dbwilson.com}{David B. Wilson}\\[-3pt]
\small Microsoft Research
\end{tabular}
}
\maketitle
\vspace*{-6pt}
\begin{abstract}
  Cover-inclusive Dyck tilings are tilings of skew Young diagrams with
  ribbon tiles shaped like Dyck paths, in which tiles are no larger
  than the tiles they cover.  These tilings arise in the study of
  certain statistical physics models and also Kazhdan--Lusztig
  polynomials.  We give two bijections between cover-inclusive Dyck
  tilings and linear extensions of tree posets.  The first bijection
  maps the statistic (area + tiles)/2 to inversions of the linear
  extension, and the second bijection maps the ``discrepancy'' between
  the upper and lower boundary of the tiling to descents of the linear
  extension.
\end{abstract}

\section{Introduction} \label{sec:intro}

Kenyon and Wilson \cite{kw} and Shigechi and Zinn-Justin \cite{MR2927185}
independently introduced the notion of ``cover-inclusive Dyck tilings'' (defined below).
The probabilities of certain events pertaining to the double-dimer model
and spanning trees are given by formulas that involve counting these Dyck tilings
\cite{kw,kw:annular}.  Dyck tilings are also relevant to the study of Kazhdan--Lusztig polynomials
\cite{MR2927185}.  More recently, Dyck tilings have arisen
in connection with fully packed loop systems \cite[Remark~2.10]{fisher-nadeau}
and other contexts \cite{fayers}.

We give two new bijections between Dyck tilings
and perfect matchings, which when restricted to Dyck tilings with a
certain ``lower path,'' become bijections to linear extensions of a
tree poset.  The first bijection is compatible with the number of
inversions of a linear extension, and gives a bijective proof of a
formula that was conjectured by Kenyon and Wilson
\cite[Conjecture~1]{kw} and proved non-bijectively by Kim \cite{kim}.
The second bijection is compatible with descents of the linear extension,
and leads to a new enumeration formula.
We also conjecture a third enumeration formula.

\subsection{Background}
Dyck paths of order $n$ are often defined as staircase lattice paths
on an $n\times n$ square grid, from the lower-left corner to the
upper-right corner, which do not go below the diagonal.  Each such
Dyck path has associated with it a Young diagram formed from the boxes
above and to the left of it.
If we rotate the lattice by $45^\circ$ and dilate it
by a factor of $\sqrt{2}$, then each step of the Dyck path is either
\begin{figure}[t]
\begin{center}
\tz{\includegraphics{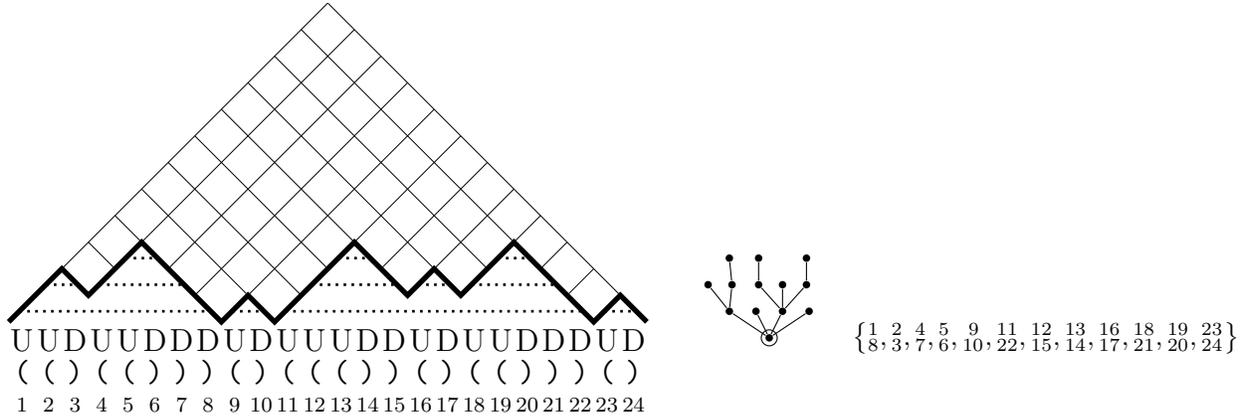}}{
\begin{tikzpicture}[rotate=45,scale=0.5]
\DyckPath[draw, line width =2pt] {2,-1,2,-3,1,-1,3,-2,1,-1,2,-3,1,-1};
\begin{scope}[shift={(19.2,-19.8)},rotate=-45,scale=0.707] % pairing
\node at (0,0) {\footnotesize $\left\{{}^{1}_{8},{}^{2}_{3},{}^{4}_{7},{}^{5}_{6},{}^{\;9}_{10},{}^{11}_{22},{}^{12}_{15},{}^{13}_{14},{}^{16}_{17},{}^{18}_{21},{}^{19}_{20},{}^{23}_{24}\right\}$};
\end{scope}
\begin{scope}[shift={(14,-14.6)},rotate=-45,scale=0.707] % tree
 \node (0) at (0,0) [vertex] {};
 \node (1) at (-1.5,1)[vertex] {};
 \node (2) at (-0.5,1) [vertex] {};
 \node (3) at (0.5,1) [vertex] {};
 \node (4) at (1.5,1) [vertex] {};
 \node (5) at (-2.3,2) [vertex] {};
 \node (6) at (-1.4,2) [vertex] {};
 \node (7) at (-.4,2) [vertex] {};
 \node (8) at (0.5,2) [vertex] {};
 \node (9) at (1.4,2) [vertex] {};
 \node (10) at (-1.5,3) [vertex] {};
 \node (11) at (-.4,3) [vertex] {};
 \node (12) at (1.4,3) [vertex] {};
 \draw (0)--(1);
 \draw (0)--(2);
 \draw (0)--(3);
 \draw (0)--(4);
 \draw (0)circle(0.3);
 \draw (1)--(5);
 \draw (1)--(6);
 \draw (3)--(7);
 \draw (3)--(8);
 \draw (3)--(9);
 \draw (6)--(10);
 \draw (7)--(11);
 \draw (9)--(12);
\end{scope}
% Young diagram
% SW/NE lines
\draw[line width=0.3pt](2,0)--(12,0);
\draw[line width=0.3pt](2,-1)--(12,-1);
\draw[line width=0.3pt](4,-2)--(12,-2);
\draw[line width=0.3pt](4,-3)--(12,-3);
\draw[line width=0.3pt](4,-4)--(12,-4);
\draw[line width=0.3pt](5,-5)--(12,-5);
\draw[line width=0.3pt](8,-6)--(12,-6);
\draw[line width=0.3pt](8,-7)--(12,-7);
\draw[line width=0.3pt](9,-8)--(12,-8);
\draw[line width=0.3pt](11,-9)--(12,-9);
\draw[line width=0.3pt](11,-10)--(12,-10);
\draw[line width=0.3pt](11,-11)--(12,-11);
% NW/SE lines
\draw[line width=0.3pt](12,0)--(12,-11);
\draw[line width=0.3pt](11,0)--(11,-11);
\draw[line width=0.3pt](10,0)--(10,-8);
\draw[line width=0.3pt](9,0)--(9,-8);
\draw[line width=0.3pt](8,0)--(8,-7);
\draw[line width=0.3pt](7,0)--(7,-5);
\draw[line width=0.3pt](6,0)--(6,-5);
\draw[line width=0.3pt](5,0)--(5,-5);
\draw[line width=0.3pt](4,0)--(4,-4);
\draw[line width=0.3pt](3,0)--(3,-1);
\draw[line width=0.3pt](2,0)--(2,-1);
%\draw[line width=0.3pt](1,0)--(1,-1);
\begin{scope}
% chords
\draw [dotted, line width=1pt](0.4,0)  -- (4,-3.6);
\draw [dotted, line width=1pt](1.4,0)  -- (2,-0.6);
\draw [dotted, line width=1pt](2.4,-1) -- (4,-2.6);
\draw [dotted, line width=1pt](3.4,-1) -- (4,-1.6);
\draw [dotted, line width=1pt](4.4,-4) -- (5,-4.6);
\draw [dotted, line width=1pt](5.4,-5) -- (11,-10.6);
\draw [dotted, line width=1pt](6.4,-5) -- (8,-6.6);
\draw [dotted, line width=1pt](7.4,-5) -- (8,-5.6);
\draw [dotted, line width=1pt](8.4,-7) -- (9,-7.6);
\draw [dotted, line width=1pt](9.4,-8) -- (11,-9.6);
\draw [dotted, line width=1pt](10.4,-8) -- (11,-8.6);
\draw [dotted, line width=1pt](11.4,-11) -- (12,-11.6);
\end{scope}
\begin{scope}[shift={(-0.1,-0.1)}]
\node at (0,-0.5){U};
\node at (0.5,-1){U};
\node at (1,-1.5){D};
\node at (1.5,-2){U};
\node at (2,-2.5){U};
\node at (2.5,-3){D};
\node at (3,-3.5){D};
\node at (3.5,-4){D};
\node at (4,-4.5){U};
\node at (4.5,-5){D};
\node at (5,-5.5){U};
\node at (5.5,-6){U};
\node at (6,-6.5){U};
\node at (6.5,-7){D};
\node at (7,-7.5){D};
\node at (7.5,-8){U};
\node at (8,-8.5){D};
\node at (8.5,-9){U};
\node at (9,-9.5){U};
\node at (9.5,-10){D};
\node at (10,-10.5){D};
\node at (10.5,-11){D};
\node at (11,-11.5){U};
\node at (11.5,-12){D};
\end{scope}
\begin{scope}[shift={(-0.7,-0.7)}]
\node at (0,-0.5){\op};
\node at (0.5,-1){\op};
\node at (1,-1.5){\cp};
\node at (1.5,-2){\op};
\node at (2,-2.5){\op};
\node at (2.5,-3){\cp};
\node at (3,-3.5){\cp};
\node at (3.5,-4){\cp};
\node at (4,-4.5){\op};
\node at (4.5,-5){\cp};
\node at (5,-5.5){\op};
\node at (5.5,-6){\op};
\node at (6,-6.5){\op};
\node at (6.5,-7){\cp};
\node at (7,-7.5){\cp};
\node at (7.5,-8){\op};
\node at (8,-8.5){\cp};
\node at (8.5,-9){\op};
\node at (9,-9.5){\op};
\node at (9.5,-10){\cp};
\node at (10,-10.5){\cp};
\node at (10.5,-11){\cp};
\node at (11,-11.5){\op};
\node at (11.5,-12){\cp};
\end{scope}
\begin{scope}[shift={(-1.3,-1.3)}]
\node at (0,-0.5){\scriptsize 1};
\node at (0.5,-1){\scriptsize 2};
\node at (1,-1.5){\scriptsize 3};
\node at (1.5,-2){\scriptsize 4};
\node at (2,-2.5){\scriptsize 5};
\node at (2.5,-3){\scriptsize 6};
\node at (3,-3.5){\scriptsize 7};
\node at (3.5,-4){\scriptsize 8};
\node at (4,-4.5){\scriptsize 9};
\node at (4.5,-5){\scriptsize 10};
\node at (5,-5.5){\scriptsize 11};
\node at (5.5,-6){\scriptsize 12};
\node at (6,-6.5){\scriptsize 13};
\node at (6.5,-7){\scriptsize 14};
\node at (7,-7.5){\scriptsize 15};
\node at (7.5,-8){\scriptsize 16};
\node at (8,-8.5){\scriptsize 17};
\node at (8.5,-9){\scriptsize 18};
\node at (9,-9.5){\scriptsize 19};
\node at (9.5,-10){\scriptsize 20};
\node at (10,-10.5){\scriptsize 21};
\node at (10.5,-11){\scriptsize 22};
\node at (11,-11.5){\scriptsize 23};
\node at (11.5,-12){\scriptsize 24};
\end{scope}
\end{tikzpicture}
}
\end{center}
\caption{A Dyck path $\lambda$ of order $n=12$ is shown in bold on the left.  Above it is its associated (rotated) Young diagram.  Below the Dyck path is its corresponding Dyck word, its balanced parentheses expression, and coordinates for the step positions.  The horizontal dotted lines depict the chords between the matching up and down steps.  In the middle is the planted plane tree corresponding to the Dyck path, where each node other than the root corresponds to a chord.  On the right is the set of chords of the Dyck path, where each chord is represented as the pair of coordinates of the matching up and down steps.  (See \cite[Exercise~6.19]{stanley}.)  A linear extension of the chord poset of $\lambda$ can be represented by placing the numbers $1,\dots,n$ on the chords of the diagram on the left (so that if one chord is nested within another, it gets a higher number), or equivalently by placing the numbers $0,\dots,n$ on the vertices of the planted plane tree (so that the numbers increase when going up the tree), or equivalently by ordering the set of chords represented as pairs (so that if one pair is nested within another, it occurs later).  See also \fref{objects}.}
\label{ud}
\end{figure}
$(+1,+1)$ (an up or ``U'' step) or $(+1,-1)$ (a down or ``D'' step).
(See \fref{ud}.)
This rotated form will be more convenient for us to work with.
A Dyck path's sequence of U and D steps, when concatenated, forms
a word which is called a Dyck word.
If the U steps are written as ``\op'' and the down steps are written
as ``\cp'', then the Dyck word is a balanced parentheses expression.

\tz{\enlargethispage{3pt}}{}
We define a \textbf{chord} of a Dyck path $\lambda$ to be the
horizontal segment between an up step and the matching down step, as
shown in \fref{ud}.  The chords of a Dyck path~$\lambda$ naturally
form a \textbf{chord poset} $P_\lambda$, where nesting is the order
relation, i.e., one chord is above another chord in the partial order
if its horizontal span lies within the span of the other chord, or
equivalently, if the \op\ and \cp\ corresponding to the first chord
are nested within the \op\ and \cp\ corresponding to the second chord.
If we adjoin a bottom-most element to the chord poset, we call the
result the tree poset (see \fref{ud}), since its Hasse diagram is a \textbf{planted plane tree},
i.e., a tree embedded in the upper half-plane with a single distinguished vertex
(the root) on the boundary of the half-plane, where two such embedded trees are considered
equivalent if their embeddings are isotopic.  Combinatorially, a planted plane tree
is a tree with a distinguished root vertex, such that the children of any vertex are ordered.
There is a
natural bijection between Dyck paths of order $n$ and planted plane
trees with $n+1$ vertices (see \cite[Exercise~6.19]{stanley}, or \fref{ud}).

\begin{figure}[t!]
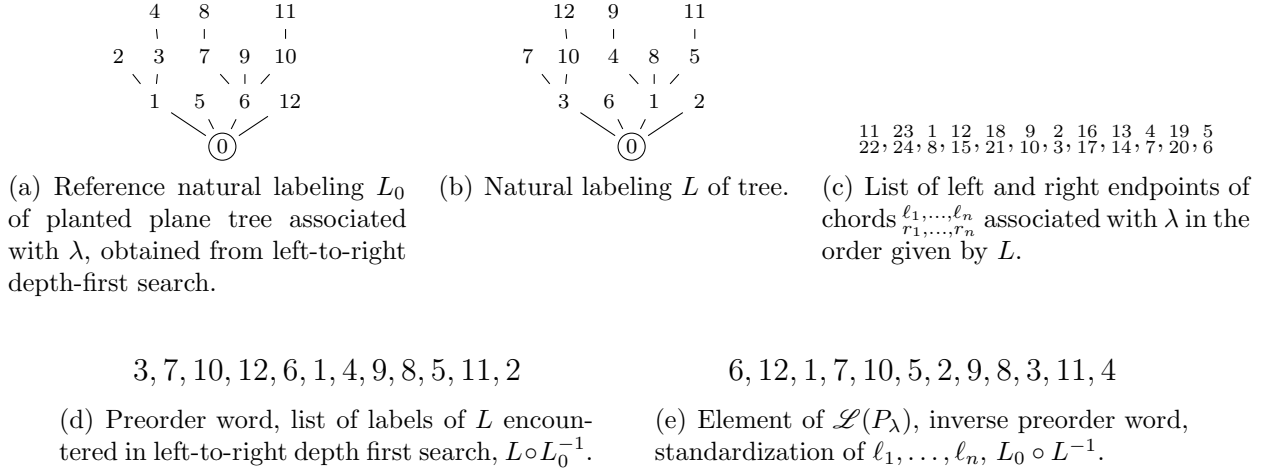

\begin{center}
\subfigure[Reference natural labeling $L_0$ of planted plane tree associated with $\lambda$, obtained from left-to-right depth-first search.]{
\hspace{24pt}
\tz{\includegraphics{dyck-tiling-figure1.pdf}}{
\begin{tikzpicture}
\begin{scope}[shift={(0,0)},scale=0.6,font=\scriptsize] % tree
 \node (0) at (0,0)  {0};
 \node (1) at (-1.5,1) {1};
 \node (2) at (-0.5,1) {5};
 \node (3) at (0.5,1) {6};
 \node (4) at (1.5,1) {12};
 \node (5) at (-2.3,2) {2};
 \node (6) at (-1.4,2) {3};
 \node (7) at (-.4,2) {7};
 \node (8) at (0.5,2) {9};
 \node (9) at (1.4,2) {10};
 \node (10) at (-1.5,3) {4};
 \node (11) at (-.4,3) {8};
 \node (12) at (1.4,3) {11};
 \draw (0)--(1);
 \draw (0)--(2);
 \draw (0)--(3);
 \draw (0)--(4);
 \draw (0)circle(0.3);
 \draw (1)--(5);
 \draw (1)--(6);
 \draw (3)--(7);
 \draw (3)--(8);
 \draw (3)--(9);
 \draw (6)--(10);
 \draw (7)--(11);
 \draw (9)--(12);
\end{scope}
\end{tikzpicture}
}
\hspace{24pt}\rule[-10pt]{0pt}{0pt}
}
\hfill
\subfigure[Natural labeling $L$ of tree.]{
\hspace{18pt}
\tz{\includegraphics{dyck-tiling-figure2.pdf}}{
\begin{tikzpicture}
\begin{scope}[shift={(10,0)},scale=0.6,font=\scriptsize] % tree
 \node (0) at (0,0)  {0};
 \node (1) at (-1.5,1) {3};
 \node (2) at (-0.5,1) {6};
 \node (3) at (0.5,1) {1};
 \node (4) at (1.5,1) {2};
 \node (5) at (-2.3,2) {7};
 \node (6) at (-1.4,2) {10};
 \node (7) at (-.4,2) {4};
 \node (8) at (0.5,2) {8};
 \node (9) at (1.4,2) {5};
 \node (10) at (-1.5,3) {12};
 \node (11) at (-.4,3) {9};
 \node (12) at (1.4,3) {11};
 \draw (0)--(1);
 \draw (0)--(2);
 \draw (0)--(3);
 \draw (0)--(4);
 \draw (0)circle(0.3);
 \draw (1)--(5);
 \draw (1)--(6);
 \draw (3)--(7);
 \draw (3)--(8);
 \draw (3)--(9);
 \draw (6)--(10);
 \draw (7)--(11);
 \draw (9)--(12);
\end{scope}
\end{tikzpicture}
}
\hspace{18pt}\rule[-10pt]{0pt}{0pt}
}
\hfill
\subfigure[List of left and right endpoints of chords ${}^{\ell_1,\dots,\ell_n}_{r_1,\dots,r_n}$ associated with $\lambda$ in the order given by $L$.]{
\footnotesize \quad$
{}^{11}_{22},
{}^{23}_{24},
{}^{1}_{8},
{}^{12}_{15},
{}^{18}_{21},
{}^{\;9}_{10},
{}^{2}_{3},
{}^{16}_{17},
{}^{13}_{14},
{}^{4}_{7},
{}^{19}_{20},
{}^{5}_{6}
$\quad\rule[-10pt]{0pt}{0pt}
}
\end{center}
%\\[12pt]
\begin{center}
\hfill
\subfigure[Preorder word, list of labels of $L$ encountered in left-to-right depth first search, $L \circ L_0^{-1}$.]{
\hspace{24pt}$
3,7,10,12,6,1,4,9,8,5,11,2
$\hspace{24pt}\rule[-10pt]{0pt}{0pt}
}
\hfill
\subfigure[Element of $\L$, inverse preorder word, standardization of $\ell_1,\dots,\ell_n$, $L_0 \circ L^{-1}$.]{
\hspace{24pt}$
6,12,1,7,10,5,2,9,8,3,11,4
$\hspace{24pt}\rule[-10pt]{0pt}{0pt}
}
\hfill \rule{0pt}{0pt}
\end{center}
\caption{Natural labelings of a planted plane tree and their associated permutations.}
\label{extension}
\end{figure}

\tz{\enlargethispage{4pt}}{}
A \textbf{natural labeling} of a poset $P$ with $n$ elements is an
order-preserving bijection $L:P\to [n]$, where $[n]$ denotes $\{1,2,\ldots,n\}$.
For the tree poset
associated with $\lambda$, it is more convenient to take a natural
labeling of the chord poset $P_\lambda$, and then label the root by $0$, as
shown in \fref{extension}.  A planted plane tree
with a natural labeling is called an \textbf{increasing planted plane tree}.
As is well known, there are $(2n-1)!!$ increasing
planted plane trees on $n+1$ vertices (see
\cite[Corollary~1(iv)]{BFS}).  Given a labeled tree $L$, if we delete
all vertices with labels larger than $k\leq n$, the result is a
labeled tree $L^{(k)}$ on $k+1$ vertices (including the root
labeled~$0$).  Given $L^{(k-1)}$, there are $2k-1$ possible positions
where the vertex with label $k$ may be attached to the labeled tree
$L^{(k-1)}$.  (Each time a new vertex gets added to the tree, the
subsequent vertex has one less possible attachment location but three
new ones: just before, just after, and on top of the new vertex.)
Thus, to each labeled tree $(P_\lambda,L)$ we can derive a sequence of attachment
sites $p_1,\ldots,p_n$, where $0\leq p_i < 2i-1$.  Any such sequence
determines a poset $P_\lambda$ together with a natural labeling $L$ of
$P_\lambda$, and the map from sequences to pairs $(P_\lambda,L)$ is a
bijection.
In terms of the endpoints of the chords, this sequence of insertion locations is given by
$$p_i = \#\{j<i \; : \; \ell_j < \ell_i\} + \#\{j<i \; : \; r_j < \ell_i \} ,$$
where $\ell_i$ and $r_i$ denote the left and right endpoints
of the chord labeled $i$, as in \frefs{extension}{~c}.

Let $L_0$ be the natural labeling of $P_\lambda$ which orders the
chords by their left endpoints (\frefs{extension}{~a}).  The
\textbf{preorder word} of a natural labeling $L$ of $P_\lambda$ is $L
\circ L_0^{-1}$, which is the permutation on $[n]$ obtained by reading
the labels of $L$ (excluding $0$) in a left-to-right depth-first
order (\frefs{extension}{~d}).  The inverse of the preorder word, $\sigma=L_0 \circ L^{-1}$,
turns out to be a more natural object.  It can also be obtained as the
``standardization'' of the sequence of left-endpoints
$\ell_1,\dots,\ell_n$, i.e., $\sigma$ is the permutation on $[n]$ for
which $\sigma_i < \sigma_j$ iff $\ell_i<\ell_j$ (\frefs{extension}{~e}).
If $\omega$ is a natural labeling of the poset $P$, Stanley \cite{MR0332509} defines
$$ \LL(P,\omega) = \{\omega \circ L^{-1} : \text{$L$ is a natural labeling of $P$}\},$$
and these are sometimes called the linear extensions of the labeled poset $(P,\omega)$.
We will abbreviate $\LL(P_\lambda)=\LL(P_\lambda,L_0)$.

It is also well known that there are $(2n-1)!!$ perfect matchings on
the numbers $1,\dots,2n$.  Given the sequence $p_1,\dots,p_n$, one
natural way to associate a perfect matching $\match(p_1,\dots,p_n)$
with it is to take $\match(p_1,\dots,p_{n-1})$, increment all the
numbers that are bigger then $p_n$, and then adjoin pair $(p_n+1,2n)$.
We define the \textbf{min-word} of the matching to be the list of the
smaller item of each pair, sorted in order of the larger item in each
pair, see \fref{objects}.

Given two Dyck paths $\lambda$ and $\mu$ of order $n$, if the path
$\mu$ is at least as high as the path $\lambda$ at each horizontal
position, then we write $\mu \subset \lambda$ (since the Young diagram
associated with $\mu$ is a subset of $\lambda$'s Young diagram), and we
write $\lambda/\mu$ for the skew Young diagram which consists of the
boxes between $\lambda$ and $\mu$.

\textbf{Dyck tiles}, also called ``Dyck strips'' in \cite{MR2927185},
are ribbon tiles (connected skew shapes that do
not contain a $2\times 2$ rectangle) in which the leftmost and
rightmost boxes are at the same height, and no box within the tile is
below these endpoints. (If each vertex in a Dyck path is replaced with
a box, and the boxes are glued together, then the result is a Dyck
tile, which explains the terminology.)  A tiling of a skew Young
diagram by Dyck tiles is a \textbf{Dyck tiling}.  We say that one Dyck
tile covers another Dyck tile if the first tile has at least one box
whose center lies straight above the center of a box in the second
tile. A Dyck tiling is called \textbf{cover-inclusive} if for each
pair of its tiles, when the first tile covers the second tile, then
the horizontal extent of the first tile is included as a subset within
the horizontal extent of the second tile.  We denote by $\D(\lambda,\mu)$
the set of all cover-inclusive Dyck tilings of shape $\lambda/\mu$, and let
$$\D(\lambda,*)=\bigcup_{\mu} \D(\lambda/\mu).$$
\fref{tilings} shows all the cover-inclusive Dyck tilings of a certain skew shape.

\begin{figure}[t!]
\begin{center}
\tz{\includegraphics[width=\textwidth]{skew-big-tile.pdf}}{\includegraphics[width=\textwidth]{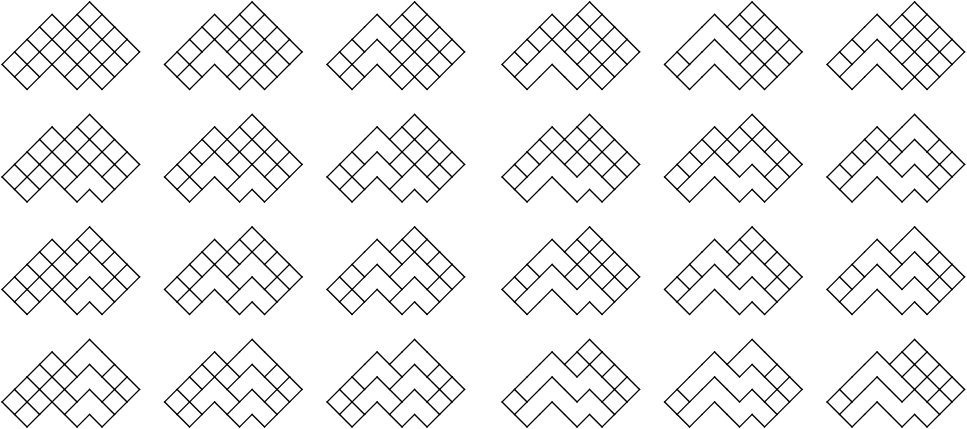}}
\caption{All the cover-inclusive Dyck tilings of a particular skew shape.  (This figure
  first appeared in \cite{kw}.)
  The generating function for the tilings of this skew shape by number of tiles is
  $t^7 + 2 t^9 + 4 t^{11} + 5 t^{13} + 5 t^{15} + 4 t^{17} + 2 t^{19} + t^{21}$, which, as
  discussed in \cite{MR2927185}, is closely related to a Kazhdan--Lusztig polynomial.
}
\label{tilings}
\end{center}
\end{figure}

\subsection{Connections between Dyck tilings and increasing trees}
It was observed empirically that there is a close connection between
Dyck tilings and linear extensions of planted plane trees.  More specifically, for a Dyck
path $\lambda$ of order $n$, the total number of cover-inclusive Dyck tilings of skew
shape $\lambda/\mu$ for some $\mu$ was conjectured \cite[Conjecture~1]{kw}
and subsequently proven \cite{kim} to be
\begin{equation} \label{hook}
  \left| \D(\lambda,*) \right| = \frac{n!}{\prod_{\text{chords $c$ of $\lambda$}} |c|},
\end{equation}
where $|c|$ is the length of the chord $c$ as measured in terms of
number of up steps in $\lambda$ between and including its ends.
This formula has the form of the tree hook-length formula
of Knuth~\cite[\old{Exercise~20, }pg.~70]{knuth3} for the
number of linear extensions of the tree poset $P_\lambda$.
These formulas call out for a bijection between cover-inclusive Dyck
tilings and linear extensions.  Here we give two such bijections.
These bijections were in part inspired by a bijection due to Aval,
Boussicault, and Dasse-Hartaut \cite{abdh} (see also \cite{abn}),
to which one of our bijections specializes in the case where $P_\lambda$ is the antichain.
Our bijections actually provide refined enumeration formulas,
which relate statistics (defined below) on the Dyck tilings to
well-studied statistics on the permutations.

For two Dyck paths $\lambda$ and $\mu$, let
$\mathbf{dis}(\lambda,\mu)$ be the ``\underline{dis}crepancy'' between
$\lambda$ and $\mu$, i.e., the number of locations where $\lambda$ has
a down step while $\mu$ has an up step (which also equals half the
number of locations where $\lambda$ and $\mu$ step in opposite
directions, i.e., half the Hamming \underline{dis}tance between the
Dyck words of $\lambda$ and $\mu$).  For a skew shape $\lambda/\mu$ we
define $\dis(\lambda/\mu)=\dis(\lambda,\mu)$.

For a Dyck tiling $T$ of shape $\sh(T)=\lambda/\mu$, we define
$\dis(T)=\dis(\lambda,\mu)$, and $\mathbf{tiles}(T)$ to be the number
of tiles in $T$, and $\mathbf{area}(T)$ to be the number of boxes of the skew
shape $\lambda/\mu$.  We define
$$\art(T)=(\text{\underline{ar}ea}(T)+\text{\underline{t}iles}(T))/2.$$
The art statistic is always integer-valued since each tile has odd
area, and appears to be more natural than the area statistic.

\old{
\newcommand{\YOminword}{7}
\newcommand{\YOminwordT}{5.5}
\newcommand{\YOcsequence}{4}
\newcommand{\YOpairing}{2.5}
\newcommand{\YOsequence}{0}
\newcommand{\YOlabeledpath}{-3}
\newcommand{\YOlabeledtree}{-7}
\newcommand{\YOPreorder}{-10}
\newcommand{\YOPreorderi}{-12.5}
\newcommand{\YOLrinorder}{-13.9}
\newcommand{\YOlRinorder}{-15.0}
\newcommand{\YOtiling}{-19.1}
\newcommand{\YOtilingsDelta}{-3.6}
\newcommand{\YOdescents}{-25.6}
\newcommand{\YOinversions}{-28.4}
\newcommand{\YOnest}{-31.2}

\newcommand{\YLminword}{6.6}
\newcommand{\YLmatching}{2.5}
\newcommand{\YLsequence}{0}
\newcommand{\YLlabeledpath}{-3}
\newcommand{\YLlabeledtree}{-6.5}
\newcommand{\YLPreorder}{-10}
\newcommand{\YLPreorderi}{-12.0}
\newcommand{\YLLrinorder}{-13.9}
\newcommand{\YLlRinorder}{-15.0}
\newcommand{\YLtiling}{-18.9}
\newcommand{\YLtilings}{-22.5}
\newcommand{\YLdis}{-25}
\newcommand{\YLdes}{-26.2}
\newcommand{\YLart}{-28.4}
\newcommand{\YLtiles}{-30.6}
\newcommand{\YLnestings}{-31.8}
}

\newcommand{\YOlabeledpath}{6.0}
\newcommand{\YOlabeledtree}{2.0}
\newcommand{\YOPreorder}{-1.0}
\newcommand{\YOPreorderi}{-3.5}
\newcommand{\YOLrinorder}{-4.9}
\newcommand{\YOlRinorder}{-6.0}
\newcommand{\YOsequence}{-8.5}
\newcommand{\YOpairing}{-12.0}
\newcommand{\YOminword}{-14.1}
\newcommand{\YOminwordT}{-15.6}
\newcommand{\YOtiling}{-19.7}
\newcommand{\YOtilingsDelta}{-4.2}
\newcommand{\YOdescents}{-28.2}
\newcommand{\YOinversions}{-31.0}
\newcommand{\YOnest}{-33.8}

\newcommand{\YLlabeledpath}{6.0}
\newcommand{\YLlabeledtree}{2.5}
\newcommand{\YLPreorder}{-1.0}
\newcommand{\YLPreorderi}{-3.0}
\newcommand{\YLLrinorder}{-4.9}
\newcommand{\YLlRinorder}{-6.0}
\newcommand{\YLsequence}{-8.5}
\newcommand{\YLmatching}{-12.0}
\newcommand{\YLminword}{-14.5}
\newcommand{\YLtiling}{-19.5}
\newcommand{\YLtilings}{-23.1}
\newcommand{\YLdis}{-27.6}
\newcommand{\YLdes}{-28.8}
\newcommand{\YLart}{-31.0}
\newcommand{\YLtiles}{-33.2}
\newcommand{\YLnestings}{-34.4}

\begin{figure}[!p]
\begin{center}
\newcommand{\ud}{\mathclap{\underset{\updownarrow}{}}}
\newcommand{\di}{\mathclap{\overset{\downarrow}{\phantom{1}}}}
\newcommand{\labeledtree}[4]{
\begin{scope}[shift={(2,\YOlabeledtree)},font=\scriptsize]
 \node (O) at (0.2,0)  {0};
 \node (A) at (-.9,1) {#1};
 \node (B) at (1.5,1)  {#2};
 \node (C) at (0.6,2)  {#3};
 \node (D) at (2.4,2)  {#4};

 \draw (O)--(A);
 \draw (O)--(B);
 \draw (B)--(C);
 \draw (B)--(D);
\end{scope}
}
\newcommand{\labeledpath}[4]{
\begin{scope}[shift={(0,\YOlabeledpath)},scale=0.9,rotate=45,font=\scriptsize]
\DyckPath[draw]{1,-1,2,-1,1,-2};
\draw node at (0.5,-0.5) {#1};
\draw node at (2.5,-2.5) {#2};
\draw node at (2.5,-1.5) {#3};
\draw node at (3.5,-2.5) {#4};
\end{scope}
}
\newcommand{\pairing}[1]{
\begin{scope}[shift={(0,\YOpairing)},scale=0.9,font=\small]
 \path [start chain = arcs going right, node distance = 0pt, every node/.style = {on chain, inner xsep = 0.3pt, inner ysep=2pt}]
  \foreach \x in {1,...,8} {
   node {\footnotesize\x}
 };
 \path[every edge/.style = {draw, bend left=60}]
  \foreach \x/\y in {#1} {
   (arcs-\x.north) edge (arcs-\y.north)
  };
\end{scope}
}
\newcommand{\sequence}[1]{
\node at(2.5,\YOsequence) {\footnotesize$#1$};
}
\newcommand{\csequence}[1]{
\node at(2,\YOcsequence) {\footnotesize$#1$};
}
\newcommand{\lrinorder}[2]{
\node at (2.5,\YOLrinorder) {\footnotesize $#1$};
\node at (2.5,\YOlRinorder) {\footnotesize $#2$};
}
\newcommand{\Preorder}[1]{
\node at(2.5,\YOPreorder) {\footnotesize$#1$};
}
\newcommand{\Preorderi}[1]{
\node at (2.5,\YOPreorderi) {\footnotesize$#1$};
}
\newcommand{\descents}[1]{
\node at (2.7,\YOdescents) {$#1$};
}
\newcommand{\inversions}[1]{
\node at (2.7,\YOinversions) {$#1$};
}
\newcommand{\nest}[1]{
\node at (2.7,\YOnest) {$#1$};
}
\newcommand{\minword}[2]{
\node at (2.5,\YOminword) {\footnotesize$#1$};
\node at (2.5,\YOminwordT) {\footnotesize$#2$};
}
\newcommand{\tiling}[3]{
\begin{scope}[shift={(0,\YOtiling)},scale=0.9,rotate=45]
\DyckPath[draw]{#1};
\DyckPath[draw]{#2};
\DyckPath[draw]{#3};
\DyckPath[draw, line width =1pt]{1,-1,2,-1,1,-2};
\end{scope}
}
\newcommand{\tilings}[3]{
\begin{scope}[shift={(0,\YOtilingsDelta)}]
\tiling{#1}{#2}{#3}
\end{scope}
}

\tz{\includegraphics{dyck-tiling-figure3.pdf}}{
\begin{tikzpicture}[scale=0.98,every node/.style = {inner sep = 0.1pt}]
\begin{scope}[scale=0.3]
\draw [dotted] (0,-25.6)--(47.2,-25.6);
\end{scope}

\begin{scope}[scale=0.3,shift={(-4,0)}]
\node at (0,\YLminword) {\footnotesize min-word $w$};
\node at (0,\YLmatching) {\footnotesize matching};
\node at (0,\YLsequence) {\footnotesize $p_1,\ldots,p_n$};
\node at (-0.5,\YLtiling) {\footnotesize Dyck tiling $\DTr$\ \ };
\node at (-0.5,\YLtilings) {\footnotesize Dyck tiling $\DTs$\ \ };
\node at (0,\YLlabeledpath) {\footnotesize labeled chords};
\node at (0,\YLlabeledtree) {\footnotesize labeled tree};
\node at (0,\YLPreorder) {\footnotesize preorder word};
\node at (0,\YLPreorderi) {\footnotesize inverse-preorder $\sigma$};
\node at (0,\YLLrinorder) {\footnotesize left-endpoints $\ell$};
\node at (0,\YLlRinorder) {\footnotesize right-endpoints $r$};
\node at (0,\YLdis) {\footnotesize $\des(\sigma)=\dis(\DTr)$};
\node at (0,\YLdes) {\footnotesize \ \ \ \ \ $=\des(w)$};
\node at (0,\YLart) {\footnotesize $\inv(\sigma)=\art(\DTs)$};
\node at (0,\YLtiles) {\footnotesize $\inv(w)=\tiles(\DTs)$};
\node at (0,\YLnestings) {\footnotesize $=\operatorname{nestings}(\text{matching})$};
\end{scope}

\begin{scope}[scale=0.3]
\minword{1346\phantom{\di}}{2578}
\pairing{1/2,3/5,4/7,6/8}
\sequence{0,2,3,5}
%\csequence{0,1,1,2}
\labeledpath{1}{2}{3}{4}
\labeledtree{1}{2}{3}{4}
\Preorder{1234}
\Preorderi{1234\phantom{\ud}}
\lrinorder{1346}{2857}
\descents{0}
\inversions{0}
\nest{0}
\tiling{1,-1,2,-1,1,-2}{1,-1,2,-1,1,-2}{1,-1,2,-1,1,-2}
\tilings{1,-1,2,-1,1,-2}{1,-1,2,-1,1,-2}{1,-1,2,-1,1,-2}
\end{scope}

\begin{scope}[scale=0.3,shift = {(6,0)}]
\minword{135\di4}{2678}
\pairing{1/2,3/6,4/8,5/7}
\sequence{0,2,3,3}
%\csequence{0,1,1,0}
\labeledpath{1}{2}{4}{3}
\labeledtree{1}{2}{4}{3}
\Preorder{1243}
\Preorderi{124\ud3}
\lrinorder{1364}{2875}
\descents{1}
\inversions{1}
\nest{1}
\tiling{1,-1,2,-1,1,-2}{1,-1,2,-1,1,-2}{1,-1,3,-3}
\tilings{1,-1,2,-1,1,-2}{1,-1,2,-1,1,-2}{1,-1,3,-3}
\end{scope}

\begin{scope}[scale=0.3,shift={(12,0)}]
\minword{2\di146}{3578}
\pairing{1/5,2/3,4/7,6/8}
\sequence{0,0,2,5}
%\csequence{0,-1,0,2}
\labeledpath{2}{1}{3}{4}
\labeledtree{2}{1}{3}{4}
\Preorder{2134}
\Preorderi{2\ud134}
\lrinorder{3146}{8257}
\descents{1}
\inversions{1}
\nest{1}
\tiling{1,-1,2,-1,1,-2}{1,-1,2,-1,1,-2}{2,-1,1,-1,1,-2}
\tilings{1,-1,2,-1,1,-2}{1,-1,2,-1,1,-2}{2,-1,1,-1,1,-2}
\end{scope}

\begin{scope}[scale=0.3,shift={(18,0)}]
\minword{2\di15\di4}{3678}
\pairing{1/6,2/3,4/8,5/7}
\sequence{0,0,3,3}
%\csequence{0,-1,1,0}
\labeledpath{2}{1}{4}{3}
\labeledtree{2}{1}{4}{3}
\Preorder{2143}
\Preorderi{2\ud14\ud3}
\lrinorder{3164}{8275}
\descents{2}
\inversions{2}
\nest{2}
\tiling{1,-1,2,-1,1,-2}{1,-1,2,-1,1,-2}{2,-1,2,-3}
\tilings{1,-1,2,-1,1,-2}{1,-1,2,-1,1,-2}{2,-1,2,-3}
\end{scope}

\begin{scope}[scale=0.3,shift={(24,0)}]
\minword{23\di16}{4578}
\pairing{1/7,2/4,3/5,6/8}
\sequence{0,1,0,5}
%\csequence{0,0,-2,2}
\labeledpath{3}{1}{2}{4}
\labeledtree{3}{1}{2}{4}
\Preorder{3124}
\Preorderi{23\ud14}
\lrinorder{3416}{8527}
\descents{1}
\inversions{2}
\nest{2}
\tiling{1,-1,2,-1,1,-2}{2,-1,1,-1,1,-2}{3,-2,1,-2}
\tilings{1,-1,2,-1,1,-2}{2,-1,1,-1,1,-2}{3,-2,1,-2}
\end{scope}

\begin{scope}[scale=0.3,shift={(30,0)}]
\minword{23\di14}{5678}
\pairing{1/7,2/5,3/6,4/8}
\sequence{0,1,0,3}
%\csequence{0,0,-2,0}
\labeledpath{3}{1}{4}{2}
\labeledtree{3}{1}{4}{2}
\Preorder{3142}
\Preorderi{24\ud13}
\lrinorder{3614}{8725}
\descents{1}
\inversions{3}
\nest{2}
\tiling{1,-1,2,-1,1,-2}{2,-1,1,-1,1,-2}{4,-4}
\tilings{1,-1,2,-1,1,-2}{2,-1,1,-1,1,-2}{4,-4}
\end{scope}

\begin{scope}[scale=0.3,shift={(36,0)}]
\minword{235\di1}{4678}
\pairing{1/8,2/4,3/6,5/7}
\sequence{0,1,3,0}
%\csequence{0,0,1,-3}
\labeledpath{4}{1}{2}{3}
\labeledtree{4}{1}{2}{3}
\Preorder{4123}
\Preorderi{234\ud1}
\lrinorder{3461}{8572}
\descents{1}
\inversions{3}
\nest{3}
\tiling{2,-1,1,-1,1,-2}{3,-1,1,-3}{4,-4}
\tilings{1,-1,2,-1,1,-2}{2,-1,1,-1,1,-2}{3,-1,1,-3}
\end{scope}

\begin{scope}[scale=0.3,shift={(42,0)}]
\minword{24\di3\di1}{5678}
\pairing{1/8,2/5,3/7,4/6}
\sequence{0,1,1,0}
%\csequence{0,0,-1,-3}
\labeledpath{4}{1}{3}{2}
\labeledtree{4}{1}{3}{2}
\Preorder{4132}
\Preorderi{24\ud3\ud1}
\lrinorder{3641}{8752}
\descents{2}
\inversions{4}
\nest{4}
\tiling{1,-1,2,-1,1,-2}{2,-1,1,-1,1,-2}{3,-1,1,-3}
\tilings{2,-1,1,-1,1,-2}{3,-1,1,-3}{4,-4}
\end{scope}

\end{tikzpicture}
}
\end{center}
\caption{Objects associated with the Dyck path $\lambda=$UDUUDUDD.
Row~1 shows natural labelings of the chord poset $P_\lambda$.
Row~2 shows essentially the same thing --- natural labelings of the planted plane tree associated with $\lambda$.
Row~3 shows the labels of the planted plane tree listed in depth-first search order, where children are searched in left-to-right order.  This is equivalent to listing the chord labels in order of the left endpoint.
Row~4 shows the inverse of the permutation from the sixth row, with marks at the descents.  These are the permutations~$\sigma$ of $\L$.
Row~5 shows the left and right endpoints of the chords of $\lambda$, listed in the order of the natural labeling.
Row~6 shows the growth sequences $p_1,\dots,p_n$ which correspond to the increasing planted plane trees in row 2.
Row~7 shows perfect matchings on $1,\dots,2n$ that correspond to the sequences $p_1,\dots,p_n$ in row 6.
Row~8 shows these same perfect matchings, represented as a $2\times n$ array of numbers, where the columns are sorted, and the bottom row is sorted, together with markings at the descents in the top row.  The top row of each $2\times n$ array is the min-word~$w$.  
Rows~9 and 10 show the cover-inclusive Dyck tilings $\DTs(\lambda,\sigma)$ and $\DTr(\lambda,\sigma)$ whose lower path is~$\lambda$.  
The last three rows give statistics on these objects.
Row~11 gives $\des(w)=\des(\sigma)=\des(\ell)=\dis(\DTr)$.
Row~12 gives $\inv(\sigma)=\inv(\ell)=\art(\DTs)$.
Row~13 gives $\inv(w)=\tiles(\DTs)=\operatorname{nestings}(\text{matching})$.
}
\label{objects}
\end{figure}

Recall that the \textbf{inversion} statistic of a permutation $\sigma$ on $[n]$ is defined by
\[
\inv(\sigma)=\#\{(i,j): 1\leq i < j \leq n, \sigma(i)>\sigma(j)\}
\]
and the \textbf{descent} statistic is defined by
\[
\des(\sigma)=\#\{i<n:\sigma_i>\sigma_{i+1}\}.
\]

In \sref{bijections} we give our two bijections, which we call $\DTs$
and $\DTr$, which stand for ``Dyck tiling strip'' and ``Dyck tiling
ribbon'' respectively, for reasons that will become apparent.  The
functions $\DTs$ and $\DTr$ are bijections from the sequences
$p_1,\dots,p_n$ to cover-inclusive Dyck tilings of order $n$ (without
restrictions on the lower path~$\lambda$ or upper path~$\mu$).  As
discussed above, these sequences $p_1,\dots,p_n$ are also in bijection
with increasing planted plane trees, so we can write
$\DTs(\lambda,\sigma)$ and $\DTr(\lambda,\sigma)$ for the maps which
take the labeled tree defined by $\lambda$ and $\sigma$ to the
sequence $p_1,\dots,p_n$ and then to the Dyck tiling.  These three
bijections are compatible with one another in the sense that the lower
paths of $\DTs(\lambda,\sigma)$ and $\DTr(\lambda,\sigma)$ are both
$\lambda$.  Thus, for a given Dyck path $\lambda$ of order $n$, the
maps $\DTs(\lambda,\cdot)$ and $\DTr(\lambda,\cdot)$ are bijections
from $\LL(P_\lambda)$ to $\D(\lambda,*)$, see \fref{objects}.
Furthermore, if $\sigma\in\LL(P_\lambda)$, then
$$\art(\DTs(\lambda,\sigma)) = \inv(\sigma)$$
and $$\dis(\DTr(\lambda,\sigma)) = \des(\sigma).$$
It is natural to ask if there is a bijection from linear extensions in
$\L$ to cover-inclusive Dyck tilings with lower path~$\lambda$ that
simultaneously maps $\inv$ to $\art$ and $\des$ to $\dis$, but as can
be seen from \fref{objects}, no such bijection exists.

Bj\"orner and Wachs \cite{BW} gave the following $q$-analog of Knuth's tree hook-length formula:
\[
  \sum_{\sigma\in \LL(P_\lambda)} q^{\inv(\sigma)}
= \frac{\qint{n}!}{\prod_{\text{vertices $v\in P_\lambda$}} \qint{|\text{subtree rooted at $v$}|}},
\]
where $[n]_q=1+q+\cdots+q^{n-1}$ and $[n]_q!=[1]_q \cdots [n]_q$.
Using this formula together with the $\DTs$ bijection, we get a
bijective proof of the following theorem, originally proven in
\cite{kim} using inductive computation:
\begin{theorem}\cite[Conjecture 1]{kw}\cite{kim}\label{thm:1}
  Given a Dyck path $\lambda$ of order $n$, we have
\begin{equation} \label{qart}
\sum_{\text{Dyck tilings $T\in\D(\lambda,*)$}} q^{\art(T)}
= \frac{\qint{n}!}{\prod_{\text{chords $c$ of $\lambda$}} \qint{|c|}}.
\end{equation}
\end{theorem}

In turn, the $\DTr$ bijection implies
\begin{theorem}\label{descents}
\begin{equation}\label{zdis}
\sum_{\text{Dyck tilings $T\in\D(\lambda,*)$}}
 z^{\dis(T)} = \sum_{\sigma\in \L} z^{\des(\sigma)}.
\end{equation}
\end{theorem}

It is evident from the form of \eqref{qart} that if $\lambda_1$ and $\lambda_2$ are two Dyck paths
whose corresponding trees are isomorphic when we ignore their embeddings in the plane, then
$$ \sum_{T\in\D(\lambda_1,*)} q^{\art(T)} = \sum_{T\in\D(\lambda_2,*)} q^{\art(T)}.$$
The corresponding fact is not obvious for the $\dis$ statistic, but Stanley \cite[Thm.~9.1 and Prop.~14.1]{MR0332509}
proved that for any naturally labeled poset $(P,\omega)$, $$\sum_{\sigma\in\LL(P,\omega)}z^{\des(\sigma)}$$
is independent of the labeling $\omega$, so then it follows from
\eqref{zdis} that
$$ \sum_{T\in\D(\lambda_1,*)} z^{\dis(T)} = \sum_{T\in\D(\lambda_2,*)} z^{\dis(T)}.$$
These sums can be computed recursively using \cite[12.6(ii) and 12.2]{MR0332509}.

Experimentally the tiles statistic also appears to behave nicely in this fashion:
\begin{conjecture}
If $\lambda_1$ and $\lambda_2$ are two Dyck paths
whose corresponding trees are isomorphic when we ignore their embeddings in the plane, then
$$ \sum_{T\in\D(\lambda_1,*)} t^{\tiles(T)} = \sum_{T\in\D(\lambda_2,*)} t^{\tiles(T)}.$$
\end{conjecture}
We have confirmed this equation by direct computation for all Dyck paths $\lambda_1$ and $\lambda_2$
of order~$8$ or less.

In contrast, the $\area$ statistic does not have this property, nor do
any of the joint statistics between $\art$, $\dis$, and $\tiles$.

The Dyck tilings $\DTs(p_1,\dots,p_n)$ and $\DTr(p_1,\dots,p_n)$ can
also be understood in terms of the perfect matching
$\match(p_1,\dots,p_n)$, as we discuss in \sref{bijections}.
We shall see, for example, that
$$ \dis(\DTr(p_1,\dots,p_n)) = \des(\matchmin(\match(p_1,\dots,p_n)))$$
and
$$ \tiles(\DTs(p_1,\dots,p_n)) = \inv(\matchmin(\match(p_1,\dots,p_n))).$$

An interesting special case for our bijections is when the lower path~$\lambda$ is the ``zig-zag''
path $$\zigzag_n=\text{UDUDUD}\ldots\text{UD}=(\U\rD)^n,$$ so that
$P_\lambda$ is the antichain, and $\L$ is the set of all permutations.
In this case, the bijection $\DTr$ restricts to the bijection between
permutations and Dyck tableaux in~\cite{abdh}, as discussed in \sref{tiling/tableau}.

We show in \sref{tiling/mad} that when $\lambda=\zigzag_n$,
$$ \art(\DTr(\zigzag_n,\sigma)) = \mad(\sigma),$$
where $\mad$ is a Mahonian statistic on permutations defined
in~\cite{CSZ} (which we review later).  The composition of the two
Dyck tiling bijections
$\DTs^{-1}(\zigzag_n,\cdot)\circ\DTr(\zigzag_n,\cdot)$ gives a
bijection mapping $\mad$ to $\inv$ which is different than the one
given in \cite{CSZ}.

If we further restrict to the case where both $\lambda=\zigzag_n$ and
all the Dyck tiles have size~1, then the Dyck tiling is determined by
its upper path $\mu$.  In \sref{tiling/231} we show that in this case,
both $\DTs(\zigzag_n,\cdot)$ and
$\DTr(\zigzag_n,\cdot)$ restrict to (classical) bijections from
permutations avoiding the pattern $231$ to Dyck paths.

\section{Dyck tiling bijections} \label{bijections}

\subsection{Bijections with increasing trees}
Recall that $\lambda/\mu$ denotes the skew shape between the lower
Dyck path $\lambda$ and upper Dyck path~$\mu$.  We choose coordinates
so that when $\lambda$ and $\mu$ are of order $n$, they start at
$(-n,0)$, each step is either $(+1,+1)$ or $(+1,-1)$, and the ending
location is $(+n,0)$.  We coordinatize a Dyck tile by the Dyck path
formed from the lower corners of the boxes contained within the Dyck
tile.  Column $s\in\Z$ refers to the set of points $(s,y)$.

Given a Dyck path $\rho$ and a column $s$, we define the
\textbf{spread} of $\rho$ at $s$ to be the Dyck path~$\rho'$ whose
points are
$$
\begin{aligned}
 \{ (x,y)&+(-1,0)&\!\!:\,(x,y)\in\rho\text{ and }x\leq s\}& \ \cup\\
 \{ (x,y)&+(0,1)&\!\!:\,(x,y)\in\rho\text{ and }x=s\}& \ \cup\\
 \{ (x,y)&+(+1,0)&\!\!:\,(x,y)\in\rho\text{ and }x\geq s\}&
 .
\end{aligned}
$$
Notice that the spread of $\rho$ at $s$ makes sense whether or not Dyck path $\rho$ overlaps column~$s$.

Given a Dyck path $\rho'$ and a column $s$, we define the
\textbf{contraction} of $\rho'$ at $s$ to be the Dyck path $\rho$
whose spread at $s$ is $\rho'$.  Not all Dyck paths $\rho'$ will have
a contraction at $s$.  If there is a contraction at $s$, then it is
unique.

We define the spread and contraction of a Dyck tile at column $s$ by
taking the spread or contraction of the Dyck path that coordinatizes
the Dyck tile.  We define the spread and contraction of a Dyck tiling
at column $s$ by taking the spread or contraction of the upper and
lower bounding Dyck paths as well as of all the Dyck tiles within the
tiling.

\newcommand{\harptextlength}{1.05in}
\begin{figure}[htbp]
\begin{center}
\tz{\includegraphics{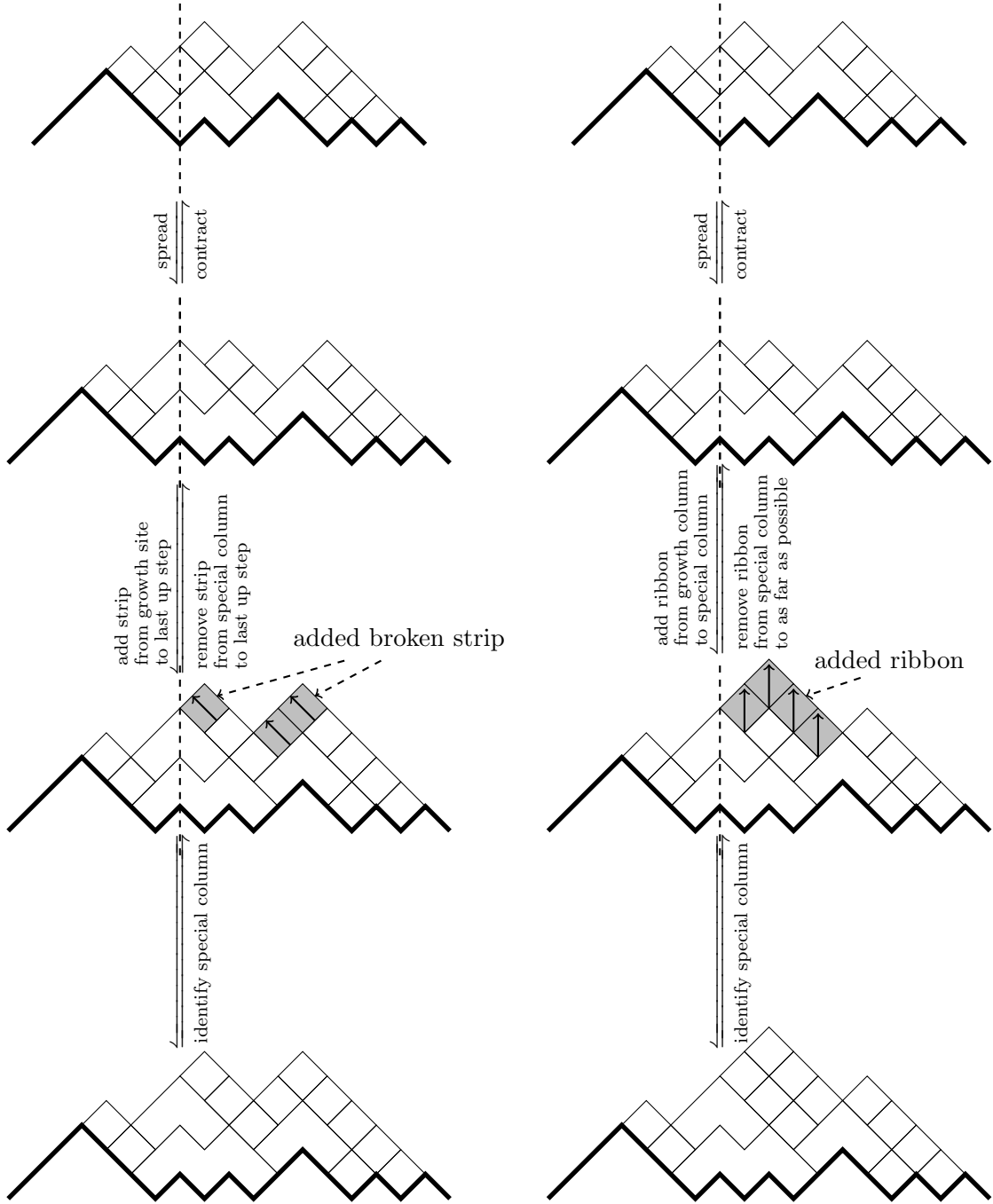}}{
\begin{tikzpicture}[rotate=45,scale=0.52]

\begin{scope}[shift={(11,-11)}]
\begin{scope}[shift={(-14.5,-13.5)}]

\node[rotate=0,fill=white] (R) at (11,-4) {\parbox{1.2in}{\small added ribbon}};

\draw[fill =lightgray] (6,-1) rectangle +(1,-1);
\draw[fill =lightgray] (7,-1) rectangle +(1,-1);
\draw[fill =lightgray] (7,-2) rectangle +(1,-1);
\draw[fill =lightgray] (7,-3) rectangle +(1,-1);

\draw[thick,->] (6,-2)--(6.9,-1.1);
\draw[thick,->] (7,-2)--(7.9,-1.1);
\draw[thick,->] (7,-3)--(7.9,-2.1);
\draw[thick,->] (7,-4)--(7.9,-3.1);

\Ribbon[shift={(6,-4)},fill=white]{2,-2};
\Ribbon[shift={(3,-2)},fill=white]{2,-1,1,-2};
\Ribbon[shift={(4,-1)},fill=white]{2,-2};

\Ribbon[shift={(7,-6)}]{1,-1};
\Ribbon[shift={(6,-3)}]{1,-1};
\Ribbon[shift={(6,-2)}]{1,-1};
\Ribbon[shift={(3,-1)}]{1,-1};
\Ribbon[shift={(3,0)}]{1,-1};
\Ribbon[shift={(8,-4)}]{1,-1};
\Ribbon[shift={(8,-5)}]{1,-1};
\Ribbon[shift={(8,-6)}]{1,-1};
\Ribbon[shift={(8,-7)}]{1,-1};

\draw [dashed,thick] (3,-4)--(7,0);

\draw[dashed,thick,->] (R)-- (8,-2.5);
\DyckPath[draw, line width =2pt]{3,-3,1,-1,1,-1,2,-2,1,-1,1,-1};

\node[rotate=90] at (1.3,-5.8){$\xleftrightharpoons[\text{identify special column}]{\text{\phantom{identify}}}$};
\end{scope}

\begin{scope}[shift={(-7,-6)}]
\draw [dashed,thick] (3,-4)--(7,0);

\Ribbon[shift={(6,-4)}]{2,-2};
\Ribbon[shift={(3,-2)},fill=white]{2,-1,1,-2};
\Ribbon[shift={(4,-1)},fill=white]{2,-2};
\Ribbon[shift={(7,-6)}]{1,-1};
\Ribbon[shift={(6,-3)}]{1,-1};
\Ribbon[shift={(6,-2)}]{1,-1};
\Ribbon[shift={(3,-1)}]{1,-1};
\Ribbon[shift={(3,0)}]{1,-1};
\Ribbon[shift={(8,-4)}]{1,-1};
\Ribbon[shift={(8,-5)}]{1,-1};
\Ribbon[shift={(8,-6)}]{1,-1};
\Ribbon[shift={(8,-7)}]{1,-1};

\draw [dashed,thick] (3,-4)--(7,0);

\DyckPath[draw, line width =2pt]{3,-3,1,-1,1,-1,2,-2,1,-1,1,-1};
\node[rotate=90] at (1.6,-5.5) {$\xleftrightharpoons[\text{\parbox{\harptextlength}{remove ribbon\\[-6pt]from special column\\[-6pt]to as far as possible}}]{\text{\parbox{\harptextlength}{add ribbon\\[-6pt]from growth column\\[-6pt] to special column}}}$};
\end{scope}

\begin{scope}[shift={(-22,-21)}]
\draw (6,-1) rectangle +(1,-1);
\draw (7,-1) rectangle +(1,-1);
\draw (7,-2) rectangle +(1,-1);
\draw (7,-3) rectangle +(1,-1);

\Ribbon[shift={(6,-4)},fill=white]{2,-2};
\Ribbon[shift={(3,-2)},fill=white]{2,-1,1,-2};
\Ribbon[shift={(4,-1)},fill=white]{2,-2};

\Ribbon[shift={(7,-6)}]{1,-1};
\Ribbon[shift={(6,-3)}]{1,-1};
\Ribbon[shift={(6,-2)}]{1,-1};
\Ribbon[shift={(3,-1)}]{1,-1};
\Ribbon[shift={(3,0)}]{1,-1};
\Ribbon[shift={(8,-4)}]{1,-1};
\Ribbon[shift={(8,-5)}]{1,-1};
\Ribbon[shift={(8,-6)}]{1,-1};
\Ribbon[shift={(8,-7)}]{1,-1};

\DyckPath[draw, line width =2pt]{3,-3,1,-1,1,-1,2,-2,1,-1,1,-1};
\end{scope}

\begin{scope}
\draw [dashed,thick] (2,-4)--(6,0);

\Ribbon[shift={(3,-2)}]{2,-2};
\Ribbon[shift={(5,-3)}]{2,-2};
\Ribbon[shift={(4,-1)}]{1,-1};

\Ribbon[shift={(6,-5)}]{1,-1};
\Ribbon[shift={(5,-2)}]{1,-1};
\Ribbon[shift={(5,-1)}]{1,-1};
\Ribbon[shift={(3,-1)}]{1,-1};
\Ribbon[shift={(3,0)}]{1,-1};
\Ribbon[shift={(7,-3)}]{1,-1};
\Ribbon[shift={(7,-4)}]{1,-1};
\Ribbon[shift={(7,-5)}]{1,-1};
\Ribbon[shift={(7,-6)}]{1,-1};

\DyckPath[draw, line width =2pt]{3,-3,1,-1,2,-2,1,-1,1,-1};
\node[rotate=90] at (1,-5){$\xleftrightharpoons[\text{contract}]{\text{spread}}$};
\end{scope}
\end{scope}
\begin{scope}
\begin{scope}[shift={(-14.5,-13.5)}]

\draw[fill =lightgray] (6,-1) rectangle +(1,-1);
\draw[fill =lightgray] (7,-3) rectangle +(1,-1);
\draw[fill =lightgray] (8,-3) rectangle +(1,-1);
\draw[thick,->] (6.5,-2)--(6.5,-1);
\draw[thick,->] (7.5,-4)--(7.5,-3);
\draw[thick,->] (8.5,-4)--(8.5,-3);

\node[rotate=0,fill=white] (R) at (12.3,-4.5) {\parbox{1.5in}{\small added broken strip}};
\draw[dashed,thick,->] (R)-- (7,-1.5);
\draw[dashed,thick,->] (R)-- (9,-3.5);

\Ribbon[shift={(6,-4)},fill=white]{2,-2};
\Ribbon[shift={(3,-2)},fill=white]{2,-1,1,-2};
\Ribbon[shift={(4,-1)},fill=white]{2,-2};

\Ribbon[shift={(7,-6)}]{1,-1};
\Ribbon[shift={(6,-3)}]{1,-1};
\Ribbon[shift={(6,-2)}]{1,-1};
\Ribbon[shift={(3,-1)}]{1,-1};
\Ribbon[shift={(3,0)}]{1,-1};
\Ribbon[shift={(8,-4)}]{1,-1};
\Ribbon[shift={(8,-5)}]{1,-1};
\Ribbon[shift={(8,-6)}]{1,-1};
\Ribbon[shift={(8,-7)}]{1,-1};

\DyckPath[draw, line width =2pt]{3,-3,1,-1,1,-1,2,-2,1,-1,1,-1};

\node[rotate=90] at (1.3,-5.8){$\xleftrightharpoons[\text{identify special column}]{\text{\phantom{identify}}}$};
\draw [dashed,thick] (3,-4)--(7,0);
\end{scope}

\begin{scope}[shift={(-7,-6)}]
\draw [dashed,thick] (3,-4)--(7,0);

\Ribbon[shift={(6,-4)}]{2,-2};
\Ribbon[shift={(3,-2)},fill=white]{2,-1,1,-2};
\Ribbon[shift={(4,-1)},fill=white]{2,-2};
\Ribbon[shift={(7,-6)}]{1,-1};
\Ribbon[shift={(6,-3)}]{1,-1};
\Ribbon[shift={(6,-2)}]{1,-1};
\Ribbon[shift={(3,-1)}]{1,-1};
\Ribbon[shift={(3,0)}]{1,-1};
\Ribbon[shift={(8,-4)}]{1,-1};
\Ribbon[shift={(8,-5)}]{1,-1};
\Ribbon[shift={(8,-6)}]{1,-1};
\Ribbon[shift={(8,-7)}]{1,-1};

\draw [dashed,thick] (3,-4)--(7,0);

\DyckPath[draw, line width =2pt]{3,-3,1,-1,1,-1,2,-2,1,-1,1,-1};
\node[rotate=90] at (1.2,-5.9){$\xleftrightharpoons[\text{\parbox{\harptextlength}{remove strip\\[-6pt]from special column\\[-6pt]to last up step}}]{\text{\parbox{\harptextlength}{add strip\\[-6pt]from growth site\\[-6pt] to last up step}}}$};
\end{scope}

\begin{scope}[shift={(-22,-21)}]
\draw (6,-1) rectangle +(1,-1);
\draw (7,-3) rectangle +(1,-1);
\draw (8,-3) rectangle +(1,-1);

\Ribbon[shift={(6,-4)},fill=white]{2,-2};
\Ribbon[shift={(3,-2)},fill=white]{2,-1,1,-2};
\Ribbon[shift={(4,-1)},fill=white]{2,-2};

\Ribbon[shift={(7,-6)}]{1,-1};
\Ribbon[shift={(6,-3)}]{1,-1};
\Ribbon[shift={(6,-2)}]{1,-1};
\Ribbon[shift={(3,-1)}]{1,-1};
\Ribbon[shift={(3,0)}]{1,-1};
\Ribbon[shift={(8,-4)}]{1,-1};
\Ribbon[shift={(8,-5)}]{1,-1};
\Ribbon[shift={(8,-6)}]{1,-1};
\Ribbon[shift={(8,-7)}]{1,-1};

\DyckPath[draw, line width =2pt]{3,-3,1,-1,1,-1,2,-2,1,-1,1,-1};
\end{scope}

\begin{scope}
\draw [dashed,thick] (2,-4)--(6,0);

\Ribbon[shift={(3,-2)}]{2,-2};
\Ribbon[shift={(5,-3)}]{2,-2};
\Ribbon[shift={(4,-1)}]{1,-1};

\Ribbon[shift={(6,-5)}]{1,-1};
\Ribbon[shift={(5,-2)}]{1,-1};
\Ribbon[shift={(5,-1)}]{1,-1};
\Ribbon[shift={(3,-1)}]{1,-1};
\Ribbon[shift={(3,0)}]{1,-1};
\Ribbon[shift={(7,-3)}]{1,-1};
\Ribbon[shift={(7,-4)}]{1,-1};
\Ribbon[shift={(7,-5)}]{1,-1};
\Ribbon[shift={(7,-6)}]{1,-1};

\DyckPath[draw, line width =2pt]{3,-3,1,-1,2,-2,1,-1,1,-1};
\node[rotate=90] at (1,-5){$\xleftrightharpoons[\text{contract}]{\text{spread}}$};
\end{scope}
\end{scope}
\end{tikzpicture}
}
\end{center}
\caption{The growth process for the $\DTs$ bijection (left column) and the $\DTr$ bijection (right column).  The spread and contract steps are the same for both bijections, as is the definition of the special column.  In the $\DTs$ bijection, a ``broken strip'' of one-box tiles is always added to each up step to the right of the growth site, which has the effect of pushing up-and-left the upper boundary (as indicated by the arrows).  In the $\DTr$ bijection, when the growth is at a column to the left of the special column, a ``ribbon'' of one-box tiles is added from the growth site to the special column, which has the effect of pushing up the upper boundary of the tiling (as indicated by the arrows).}
\label{columndel}
\end{figure}

\begin{figure}[htbp]
\vspace*{-24pt}
\begin{center}
\tz{\includegraphics[height=0.83\textheight]{growth-bijections.pdf}}{\includegraphics[height=0.83\textheight]{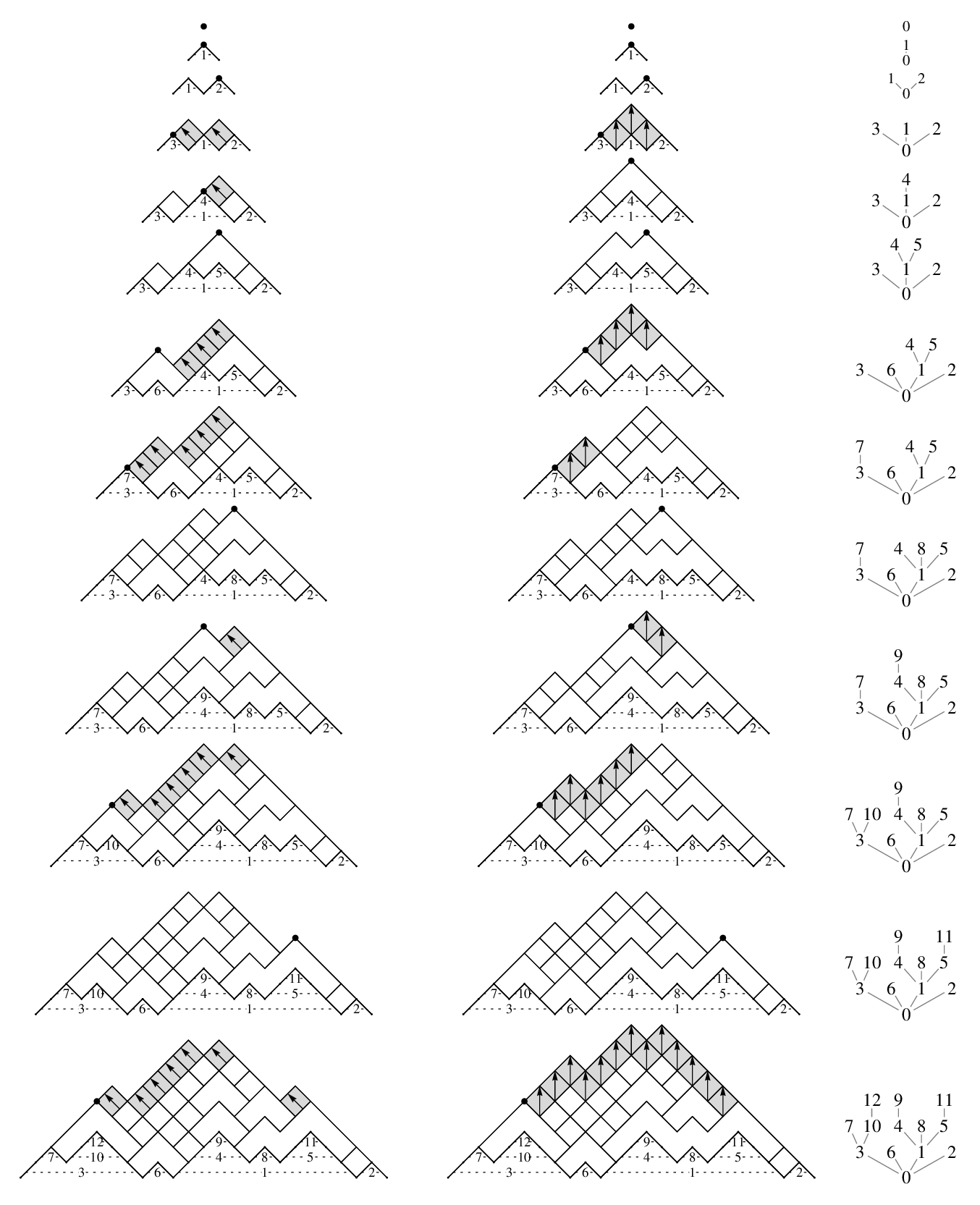}}
\end{center}
\vspace*{-18pt}
\caption{Example showing the bijections from increasing planted plane trees (linear extensions of a chord poset) to Dyck tilings, with the $\DTs$ bijection on the left, and the $\DTr$ bijection on the right.  The increasing planted plane tree is built up one vertex at a time in numerical order of the vertex labels, while the Dyck tilings are built up, starting from the empty Dyck tiling, by a sequence of the growth steps illustrated in \fref{columndel}, with newly added tiles shown in gray.  At each stage~$k$, the linear extension of the chord poset $P_{\lambda^{(k)}}$ is shown together with the $k$\th Dyck tiling, whose lower boundary is $\lambda^{(k)}$.  The preorder word of the final linear extension is $3, 7, 10, 12, 6, 1, 4, 9, 8, 5, 11, 2$, whose inverse is $6, 12, 1, 7, 10, 5, 2, 9, 8, 3, 11, 4$, which has $6$ descents and $34$ inversions.  The discrepancy between the upper and lower paths of the final tiling on the right is $6$, and $(\text{area}+\text{tiles})/2$ of the final tiling on the left is $34$.
}
\label{bijectionexample}
\end{figure}

Given a Dyck tiling $T$ of  $\lambda/\mu$, a column $s$ is \textbf{eligible} if:
\begin{enumerate}
\item The upper boundary $\mu$ contains an up step that ends in column $s$.
\item The intersection of $\mu$ with column $s$ is not the top corner of a Dyck tile of $T$ containing just one box.
\end{enumerate}
There is always at least one eligible column of a Dyck tiling $T$,
since the leftmost step of~$\mu$ ends at an eligible column.  We
define the \textbf{special column} of a Dyck tiling to be its
rightmost eligible column.

We now define two growth processes on Dyck tilings, one of which is used in the
$\DTs$ bijection, the other in the $\DTr$ bijection.  These processes
are very similar --- they both involve spreading the Dyck tiling at a
certain column $s$ and adding boxes to the right of the new column.
In the \textbf{strip-grow} process, we add a ``broken strip'' of
one-box tiles from the growth site to the right boundary of the tile,
so that a portion of $\mu$ is pushed up and left at a $45^\circ$
angle; see Figures~\ref{columndel} and~\ref{bijectionexample}.
In the \textbf{ribbon-grow}
process, we add a ``ribbon'' of one-box tiles from the growth site
right-wards to the special column, so that a portion of $\mu$ is
pushed up; see Figures~\ref{columndel} and~\ref{bijectionexample}.

Formally, given a Dyck tiling $T$ of order $n$, and a column $s$ such
that $-n\leq s \leq n$, we define \textbf{strip-grow}$(T,s)$
to be the Dyck tiling formed by first spreading $T$ at $s$
to get $T''$, and then adding a NE ``broken strip'' of one-box tiles to each
up (NE) step of the upper boundary of $T$ from the growth site until
the right boundary of $T''$ to obtain $T'=\sgrow(T,s)$.

Similarly, we define \textbf{ribbon-grow}$(T,s)$ to be
the Dyck tiling formed by first spreading~$T$ at $s$ to get $T''$,
and then if $T''$'s special column $Q$ is to the right of $s$, adding a
ribbon of one-box tiles on top of $T''$ at columns that are strictly
between columns~$s$ and~$Q$ to obtain $T'=\rgrow(T,s)$.
Notice that the upper boundary of
tiling $T''$ has a down step starting at column~$s$, and an up step
ending at column $Q$, so adding this ribbon of one-box tiles to~$T''$
results in a valid Dyck tiling $T'$.

If $T$ is cover-inclusive, then both $\sgrow(T,s)$ and $\rgrow(T,s)$
are also cover-inclusive.  The maps $\sgrow$ and $\rgrow$
are the \textbf{grow maps}.

Column $s$ is eligible for $T'$.  Because we added a strip or a
ribbon to the spread of $T$ to get $T'$, there are no eligible columns of $T'$
to the right of $s$, so $s$ is $T'$'s special column in both growth processes.

Given a cover-inclusive Dyck tiling $T'$ of order $n+1$, we now define
\textbf{strip-shrink}$(T')$ and \textbf{ribbon-shrink}$(T')$,
which we will show to be the inverses of the corresponding grow maps.
First, we identify $T'$'s special column $s$.  For the strip-shrink
map, we then remove all the one-box tiles to the right of $s$ that are
part of an up step in $T'$'s upper boundary.  This operation gives a
new cover-inclusive Dyck tiling $T''$ since the speciality of $s$
ensures that to the right of $s$ all up steps of $\mu$ are part of
one-box tiles, so removing them will keep the upper boundary a Dyck
path.  For the ribbon-shrink map, we find the column $r$ for which the
tiling $T'$ has one-box Dyck tiles on top of columns
$s+1,s+2,\dots,r-1$ but not on top of column~$r$.  (If there are no
such one-box Dyck tiles, then $r=s+1$.)  We then remove these one-box
Dyck tiles to obtain a new cover-inclusive Dyck tiling $T''$.

For both the strip-shrink and ribbon-shrink maps, since $s$ was special,
the upper boundary~$\mu'$ of $T'$ makes an up
step from $s-1$ to $s$.  Because $s+1$ was not eligible in $T'$,
either $\mu'$ makes a down step from column~$s$ to $s+1$, or it makes
an up step but there is a one-box Dyck tile on the top border,
which is then removed in $T''$.  In $T''$, the top tile in column~$s$
has a peak at $s$.  Because $T'$ is cover-inclusive, $T''$ is also
cover-inclusive, and so any tile of~$T''$ which intersects columns
$s-1$, $s$, or $s+1$ in fact intersects all three columns, and has a
peak at $s$.  Therefore $T''$ can be contracted at column~$s$ to obtain
a new cover-inclusive Dyck tiling~$T$.  The maps from $T'$ to $(T,s)$
are the \textbf{shrink maps}.

\begin{lemma} \label{lem:grow-shrink} The strip/ribbon-grow maps
  and the strip/ribbon-shrink maps are inverses and define a
  bijection from pairs $(T,s)$, where $T$ is a cover-inclusive Dyck
  tiling of order $n$ and $s$ is an integer between $-n$ and $n$
  inclusive, to cover-inclusive Dyck tilings $T'$ of order $n+1$.
\end{lemma}

\begin{proof}
  Suppose that we apply one of the grow maps to $(T,s)$ to get $T'$
  and then apply the corresponding shrink map to $T'$.  We already saw
  that $s$ is the special column of $T'$, which is then recovered by
  the shrink map.

  Consider the strip-grow followed by the strip-shrink map.  The
  addition of a one-box tile strip to the right of $s$ ensures that
  $s$ is the special column of $T'$ and thus shrinking $T'$
  results in removing that same strip of boxes.

  Consider the ribbon-grow followed by the ribbon-shrink map.  Let $q$
  denote $T$'s special column.  If $q\leq s$ then no new one-box Dyck
  tiles are added by the ribbon-grow map, column $s+1$ of $T'$ does
  not contain a one-box Dyck tile, and no one-box Dyck tiles are removed
  by the shrink map.  If $q>s$, then $T''$'s special column is $q+1$,
  and $q-s$ one-box Dyck tiles are added by the grow map, one in each
  of columns $s+1,\dots,q$.  Because $q+1$ is $T''$'s special column,
  there is no one-box Dyck tile of either $T''$ or $T'$ in column
  $q+1$.  Thus ribbon-shrink will remove precisely the one-box
  tiles that ribbon-grow added to columns $s+1,\dots,q$.

  In either the strip or ribbon cases, the shrink map removes
  precisely those one-box Dyck tiles that the grow map added.  The
  contraction of the shrink map undoes the spreading of the grow map,
  so shrinking $T'$ results in $(T,s)$.

  Next, suppose that we apply the shrink map to $T'$ to get $(T,s)$
  and then apply the grow map to $(T,s)$.  Column $s$ was $T'$'s
  special column.

  Consider the strip-shrink followed by the strip-grow map.  Since $s$
  was $T'$'s special column, all up steps of $T'$'s upper boundary to
  the right of $s$ are the top boundary of a one-box tile, which
  strip-shrink then removes.  Each of these up-steps still exist in
  $T''$ (translated by $(+1,-1)$) and in $T'$ (translated by $(0,-1)$).
  The strip-grow map then adds back the one-box tiles at these locations.

  Consider the ribbon-shrink followed by the ribbon-grow map.  Let $Q$
  denote the first column to the right of $s$ for which $T'$ does not
  have a one-box Dyck tile (such a column exists).  If $Q>s+1$, then
  upon removing the topmost one-box Dyck tiles in columns
  $s+1,\dots,Q-1$ to get cover-inclusive Dyck tiling $T''$, column $Q$
  is an eligible column of $T''$, and the rightmost such column (since
  $s<Q$ was special for $T'$), so $Q$ is $T''$'s special column.
  Upon shrinking, $Q-1>s$ is $T$'s special column, so the ribbon-grow
  map adds back in the one-box Dyck tiles to positions $s+1,\dots,Q-1$.
  Otherwise, $s=Q-1$ is $T''$'s special column, ribbon-shrink does not
  remove any one-box tiles, the special column of $T$ is $\leq s$,
  and the ribbon-grow map does not add any one-box tiles.

  In either the strip or ribbon cases, the spreading of the grow map
  undoes the contraction of the shrink map, and the grow map then adds
  one-box Dyck tiles precisely in the positions where the shrink map
  removed them, so growing $(T,s)$ results in $T'$.
\end{proof}

The bijections $\DTs$ and $\DTr$ are given by repeated application
of the strip-grow and ribbon-grow maps respectively.  More precisely,
we first do a minor change of coordinates,
$$ p_i = (i-1) + s_i,$$
so that $0\leq p_i\leq 2(i-1)$, and then define
$$
\DTs(p_1,\dots,p_n)  = \begin{cases} \varnothing & n=0, \\ \sgrow(\DTs(p_1,\dots,p_{n-1}),p_n-(n-1)) & n>0, \end{cases}
$$
and
$$
\DTr(p_1,\dots,p_n)  = \begin{cases} \varnothing & n=0, \\ \rgrow(\DTr(p_1,\dots,p_{n-1}),p_n-(n-1)) & n>0. \end{cases}
$$

\begin{theorem}
  \label{tilings_seq_bijection}
  The maps $\DTs$ and $\DTr$ are bijections from integer sequences
  $p_1,\dots,p_n$ with $0\leq p_i\leq 2(i-1)$ to cover-inclusive Dyck
  tilings of order $n$.
\end{theorem}
\begin{proof}
  Immediate from \lref{lem:grow-shrink}.
\end{proof}

\subsection{Comparison of statistics}
Next we compare the bijections $\DTs$ and $\DTr$ to the bijection from
integer sequences $p_1,\dots,p_n$ to increasing planted plane trees.
The strip-grow and ribbon-grow maps introduce a new chord in the lower
boundary of the Dyck tiling; the existing chords may be stretched, but
their relative order is unchanged.  (By induction, the lower
boundaries of $\DTs(p_1,\dots,p_n)$ and $\DTr(p_1,\dots,p_n)$ are the
same.)  If we keep track of the order in which the chords are
introduced, the result is a natural labeling $L$ of the chord poset
$P_\lambda$ of the lower boundary~$\lambda$, which together comprise
an increasing planted plane tree, as shown in
\fref{bijectionexample}.  In fact, this increasing planted plane
tree is the one given by the standard bijection from sequences
$p_1,\dots,p_n$ to increasing planted plane trees.
We can represent an increasing planted plane tree by a Dyck path
$\lambda$ and a permutation $\sigma\in\L$, as shown in \fref{extension}.
Given such a pair $(\lambda,\sigma)$, we define
$$\DTs(\lambda,\sigma) = \DTs(\text{sequence $p_1,\dots,p_n$ which yields labeled tree defined by $(\lambda,\sigma)$})$$
and
$$\DTr(\lambda,\sigma) = \DTr(\text{sequence $p_1,\dots,p_n$ which yields labeled tree defined by $(\lambda,\sigma)$}).$$
We therefore obtain the following theorem:
\begin{theorem}
  \label{extension/tiling}
  For each Dyck path $\lambda$, the maps $\DTs(\lambda,\cdot)$ and
  $\DTr(\lambda,\cdot)$ are bijections from linear extensions in $\L$
  to cover-inclusive Dyck tilings with lower path~$\lambda$.
\end{theorem}

Next we compare statistics of the Dyck tiling $T$ to statistics of the
permutation $\sigma$.  To do this, it is convenient to view an
increasing planted plane tree as having its edges labeled rather than its
vertices, and when the tree is represented by a Dyck path $\lambda$,
for the labels to reside on the up steps of $\lambda$.  The preorder
word ($\sigma^{-1}$) is the listing of these labels in order.  Both
grow maps will insert a new chord into $\lambda$ at a position $s$.
Let $n$ be the order of Dyck path $\lambda$, and let $\lambda'$ and
$\sigma'$ be the Dyck path and linear extension associated with
$\sgrow(T,s)$ or $\rgrow(T,s)$.  The word $(\sigma')^{-1}$ is just
$\sigma^{-1}$ with $n+1$ inserted at the location which is the number
of up steps of $\lambda$ to the left of $s$.

\begin{theorem}
  \label{inv/art}
  For each Dyck path $\lambda$ and linear extension $\sigma\in\L$,
  $$ \art(\DTs(\lambda,\sigma)) = \inv(\sigma).$$
\end{theorem}
\begin{proof}
By the above discussion,
$$ \inv((\sigma')^{-1}) - \inv(\sigma^{-1}) = \text{\# up steps of $\lambda$ to the right of $s$}.$$
Notice that with $T'=\sgrow(T,s)$, we have $\tiles(T')-\tiles(T)=(n-s-\mu_s)/2$, which is the number
of up steps of $\mu$ to the right of
column~$s$, and that $\area(T')-\area(T)=\mu_s-\lambda_s+(n-s-\mu_s)/2$.
(Here $\rho_s$ denotes the height of Dyck path $\rho$ in column $s$.)
In other words,
$\art(T')-\art(T) = (n-s-\lambda_s)/2$,
which we can write as
$$
\art(\sgrow(T,s))-\art(T) = \text{\# up steps of $\lambda$ to the right of $s$}.
$$
Upon combining these equations and using induction,
and using the fact that $\inv(\sigma^{-1})=\inv(\sigma)$,
the theorem follows.
\end{proof}
\begin{theorem}
  \label{des/dis}
  For each Dyck path $\lambda$ and linear extension $\sigma\in\L$,
  $$ \dis(\DTr(\lambda,\sigma)) = \des(\sigma).$$
\end{theorem}
\begin{proof}
  Let $T'=\rgrow(T,s)$.  By the above discussion, we see that the
  $n+1$ occurs before the $n$ in $\sigma'$ iff $T'$'s special column
  $s$ is smaller than $T$'s special column.  Thus $$
  \des(\sigma')-\des(\sigma) = \begin{cases} 1 & \text{$\rgrow$ map
      added a ribbon} \\ 0 & \text{otherwise.} \end{cases}$$ Notice
  that spreading $T$ at $s$ does not change the discrepancy between
  the lower and upper paths.  If the ribbon-grow map added a ribbon to
  the spread of $T$ at $s$, then $T'$ has one extra location, the
  place between $s$ and $s+1$, where the upper path $\mu'$ goes up
  while the lower path $\lambda'$ goes down.  Thus $\dis(T')-\dis(T) =
  \des(\sigma')-\des(\sigma)$.  Induction completes the proof.
\end{proof}

\begin{proof}[Proof of \tref{thm:1}]
  Immediate from \tref{extension/tiling} and \tref{inv/art} and the
  theorem of Bj\"orner and Wachs \cite{BW} on the $q$-hook-length
  formula for the $q$-distribution of the inversion statistic of
  linear extensions.
\end{proof}

\begin{proof}[Proof of \tref{descents}]
Immediate from \tref{extension/tiling} and \tref{des/dis}.
\end{proof}

\section{\textit{Histoires d'Hermite}}

Next we compare the growth of the perfect matching
$\match(p_1,\dots,p_n)$ with the Dyck tilings $\DTs(p_1,\dots,p_n)$
and $\DTr(p_1,\dots,p_n)$.  The shape of a perfect matching of
$\{1,\dots,2n\}$ is the Dyck path of order $n$ which has an up-step at
the location of the smaller item in each pair, and a down-step at the
location of the larger items.  For each pair $(a,b)$ of the matching,
we can record the number of other pairs $(c,d)$ which nest it, i.e.,
for which $c<a<b<d$.  If we record these nesting numbers on the
down-steps of the Dyck path, the resulting labeled Dyck path is called
an \textbf{\textit{histoire d'Hermite}}, and the perfect matching can
be recovered from it.  The down-steps can also be labeled according to
crossings, i.e., the number of other pairs $(c,d)$ for which
$a<c<b<d$, but for our purposes it is more convenient to work with
nestings, and to record the nesting numbers on the up steps of the
Dyck path.

\begin{figure}[htbp]
\vspace*{-24pt}
\begin{center}
\tz{\includegraphics[height=.98\textheight]{hermite.pdf}}{\includegraphics[height=.98\textheight]{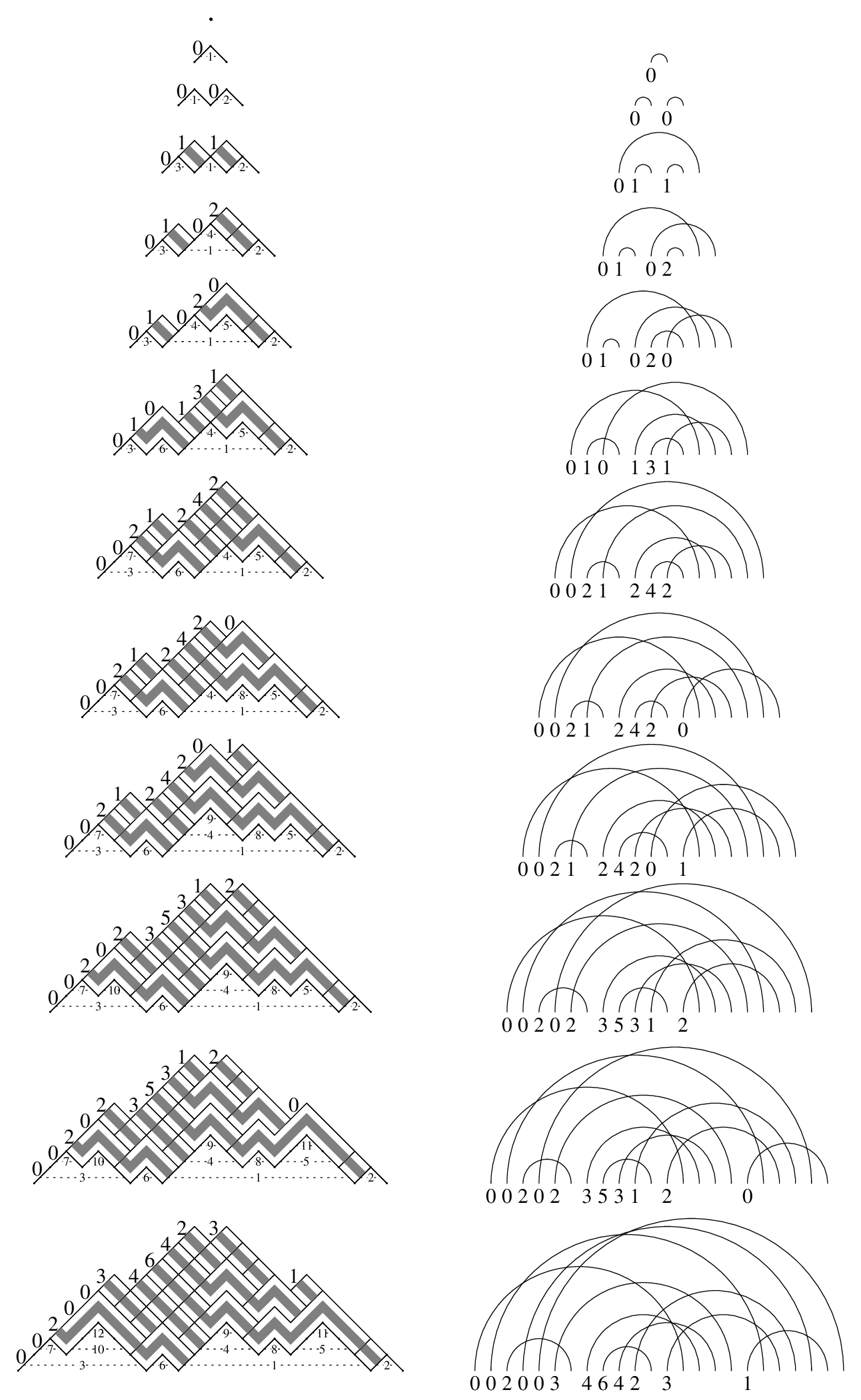}}
\end{center}
\vspace*{-18pt}
\caption{
The growth of the DTS Dyck tiling from \fref{bijectionexample} (on left) together with its corresponding perfect matching (on right).  The \textit{histoire d'Hermite\/} is shown with the perfect matching, where the numbers count nestings, and with the Dyck tiling, where the numbers count tiles.
}
\label{growth-hermite}
\end{figure}

Kim \cite{kim} and Konvalinka
showed how to transform a cover-inclusive Dyck tiling
into an \textit{histoire d'Hermite}, which we illustrate in
\fref{growth-hermite} without defining it formally.  Each number on the
up step counts the number of tiles of the tiling which are encountered
by the exploration process of gray paths (shown in
\fref{growth-hermite}), and each tile is encountered exactly once.

\begin{theorem}
  The \textit{histoire d'Hermite\/} arising from exploring
  $\DTs(p_1,\dots,p_n)$ from the left is the same as the
  \textit{histoire d'Hermite\/} arising from $\match(p_1,\ldots,p_n)$
  resulting from recording the nesting numbers on the up steps.
\end{theorem}
\begin{proof}
  The theorem is trivially true when $n=0$.  Suppose that it is true
  for $n$, and $T$ is a Dyck tiling of order $n$.  The $\sgrow$ map
  modifies the upper boundary $\mu$ of $T$ by inserting an up step at
  the special column, and appending a down step at the end.  When we
  update the perfect matching, a new smaller element is inserted at
  the location specified by $p_{n+1}$, which is matched with $2(n+1)$,
  so by induction, the upper boundary of the Dyck tiling is the shape
  of the perfect matching.  The $\sgrow$ map adds one-box tiles at
  each of the up steps to the right of the new special column, so the
  numbers in the associated \textit{histoire d'Hermite\/} after the
  special column are incremented.  In the perfect matching, after the
  new pair is added, the nesting numbers of each pair whose smaller
  element is to the right of the insertion point are incremented.
  This completes the induction.
\end{proof}

From this we see that the number of tiles in $\DTs(p_1,\dots,p_n)$
equals the nesting number of $\match(p_1,\dots,p_n)$, which in turn is
the number of inversions in $\matchmin(\match(p_1,\dots,p_n))$.

\begin{theorem}
  If $(\lambda,\sigma)$ is the labeled tree arising from $(p_1,\dots,p_n)$,
  then $$\DES(\sigma) = \DES(\matchmin(\match(p_1,\ldots,p_n))).$$
  In particular, $\dis(\DTr(p_1,\dots,p_n)) = \des(\matchmin(\match(p_1,\ldots,p_n)))$.
\end{theorem}
\begin{proof}
  Recalling the construction of the perfect matching, we see that $i$
  is a descent of the min-word when $p_{i+1}\leq p_i$.  In the
  construction of the increasing planted plane tree, this is
  precisely when the node labeled $i+1$ occurs to the left of the node
  labeled $i$ in the left-to-right depth-first search order, which
  occurs precisely when $\sigma_{i+1}<\sigma_i$.
\end{proof}

\section{Dyck tableaux}
\label{tiling/tableau}

\begin{figure}[b!]
\begin{center}
\tz{\includegraphics[scale=0.8]{chord-increasing.pdf}}{\includegraphics[scale=0.8]{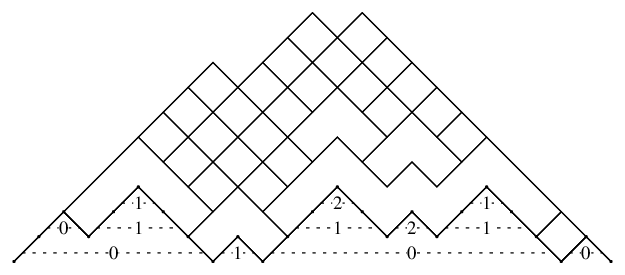}}
\hfill
\tz{\includegraphics{dyck-tiling-figure5.pdf}}{
\begin{tikzpicture}
\begin{scope}[shift={(0,0.4)},scale=0.8,font=\scriptsize] % tree
 \node (0) at (0,0)  {0};
 \node (1) at (-1.5,1) {0};
 \node (2) at (-0.5,1) {1};
 \node (3) at (0.5,1) {0};
 \node (4) at (1.5,1) {0};
 \node (5) at (-2.3,2) {0};
 \node (6) at (-1.4,2) {1};
 \node (7) at (-.4,2) {1};
 \node (8) at (0.5,2) {2};
 \node (9) at (1.4,2) {1};
 \node (10) at (-1.5,3) {1};
 \node (11) at (-.4,3) {2};
 \node (12) at (1.4,3) {1};
 \draw (0)--(1);
 \draw (0)--(2);
 \draw (0)--(3);
 \draw (0)--(4);
 \draw (0)circle(0.3);
 \draw (1)--(5);
 \draw (1)--(6);
 \draw (3)--(7);
 \draw (3)--(8);
 \draw (3)--(9);
 \draw (6)--(10);
 \draw (7)--(11);
 \draw (9)--(12);
\end{scope}
\end{tikzpicture}
}
\hfill
\tz{\includegraphics{dyck-tiling-figure6.pdf}}{
\begin{tikzpicture}
\begin{scope}[shift={(0,0.4)},scale=0.8,font=\scriptsize] % tree
 \node (0) at (0,0)  {0};
 \node (1) at (-1.5,1) {$\leq0$};
 \node (2) at (-0.5,1) {$\leq3$};
 \node (3) at (0.5,1) {$\leq1$};
 \node (4) at (1.5,1) {$\leq0$};
 \node (5) at (-2.3,2) {$\leq0$};
 \node (6) at (-1.4,2) {$\leq1$};
 \node (7) at (-.4,2) {$\leq3$};
 \node (8) at (0.5,2) {$\leq3$};
 \node (9) at (1.4,2) {$\leq1$};
 \node (10) at (-1.5,3) {$\leq1$};
 \node (11) at (-.4,3) {$\leq3$};
 \node (12) at (1.4,3) {$\leq1$};
 \draw (0)--(1);
 \draw (0)--(2);
 \draw (0)--(3);
 \draw (0)--(4);
 \draw (0)circle(0.3);
 \draw (1)--(5);
 \draw (1)--(6);
 \draw (3)--(7);
 \draw (3)--(8);
 \draw (3)--(9);
 \draw (6)--(10);
 \draw (7)--(11);
 \draw (9)--(12);
\end{scope}
\end{tikzpicture}
}
\vspace*{-12pt}
\end{center}
\caption{On the left is a cover-inclusive Dyck tiling of a certain skew shape $\lambda/\mu$, in which the chords of $\lambda$ have been labeled according to the number Dyck tiles above the chord which cover both endpoints of the chord.  Given $\lambda/\mu$ and these labels, the cover-inclusive Dyck tiling can be recovered.  These chord labels define a weakly increasing labeling of the tree poset associated with $\lambda$, as shown in the middle.  On the right is shown the maximum values of these labels for any cover-inclusive Dyck tiling of $\D(\lambda,\mu)$.}
\label{tree-poset}
\end{figure}

In this section we explain how the $\DTr$ bijection from Dyck tilings
to perfect matchings is related to the work by Aval, Boussicault, and
Dasse-Hartaut \cite{abdh} on what they call Dyck tableaux.  We make
use of a bijection from Dyck tilings of a skew shape $\lambda/\mu$
to bounded weakly increasing labelings of the planted plane tree
associated with $\lambda$, which is illustrated in \fref{tree-poset}.
This bijection appears in \cite[Prop.~1.11]{kw} in the case
$\lambda=\zigzag_n$, and in \cite[Sect.~4]{MR2927185} for general skew
shapes.  For the reader's convenience, we review the bijection in
\pref{prop:tilings-poset}.
\pagebreak

\begin{proposition}[\cite{MR2927185}] \label{prop:tilings-poset}
Let $\lambda/\mu$ be a skew shape.  For each chord $c$ of the lower
Dyck path~$\lambda$,
let $h_c$ denote the minimal thickness of the portion of the skew shape
$\lambda/\mu$ between the endpoints of the chord,
i.e., $h_c$ is the maximal number of Dyck tiles that can fit in the shape $\lambda/\mu$ and
cover chord $c$.  There is a bijection
between the cover-inclusive Dyck tilings $T$ of $\lambda/\mu$ and the
weakly increasing assignments of nonnegative integers to the poset of
chords $P_\lambda$, such that the number $g_c$ assigned to chord~$c$
satisfies $0\leq g_c\leq h_c$.  This bijection satisfies $\sum_c g_c =
(\area(T)-\tiles(T))/2$.
\end{proposition}
\begin{proof}
First observe that in any cover-inclusive Dyck tiling of skew shape
$\lambda/\mu$, every tile is shaped like the portion of the lower
boundary $\lambda$ directly beneath it.  We can assign to chord~$c$ a
number $g_c$, where $0\leq g_c \leq h_c$, which encodes the number of
tiles directly above $c$ in which the boxes in columns $\ell$ and $r$
are in the same tile, as shown in \fref{tree-poset}.  Since we are
tiling with Dyck tiles, if a chord $c'$ is above chord $c$, then
$g_{c'}\geq g_c$, so this labeling of the chord poset $P_\lambda$ is
weakly increasing.

The map from weakly increasing labelings of the chords in $P_\lambda$
to Dyck tilings is as follows.  Inductively we add Dyck tiles on top
of chords, starting from the lowest chords.  If $c_1$ covers $c_2$ in
the poset, then we add $g_{c_1}-g_{c_2}\geq 0$ new Dyck tiles whose
endpoints are exactly above the endpoints of $c_1$. By construction,
this ensures cover-inclusiveness since smaller tiles are added on top
of larger ones. By definition of $h_c$, all added tiles will fit in
$\lambda/\mu$.

It is straightforward to check that these maps are inverses.
\end{proof}

The tree structure leads to a recursive algorithm for enumerating the
number of bounded weakly increasing labelings of trees by the statistic $\sum_c g_c$,
which Lascoux and Sch\"utzenberger showed to be equivalent to
computing Kazhdan--Lusztig polynomials for pairs of Grassmannian
permutations \cite{MR646823}.
\old{
\pref{prop:tilings-poset} gives
a recursive algorithm for computing the number of Dyck tilings
of~$\lambda/\mu$.  Specifically, if we let $f_{c,v}(x)$ denote the
generating function
of bounded weakly increasing labelings of the subtree rooted at chord
$c$ in which chord $c$ is labeled $v$ or higher, enumerated according
to the sum of the labels in the subtree, then
$$
f_{c,v}(x) = \begin{cases} 0 & v>h_c \\ \displaystyle f_{c,v+1}(x) + x^v \prod_{\substack{c'\in P_\lambda\\c'>c\text{ in $P_\lambda$}}} f_{c',v}(x) & v\leq h_c. \end{cases}
$$
Then
$$ \sum_{\text{Dyck tilings $T\in\D(\lambda,\mu)$}} x^{(\area(T)-\tiles(T))/2} = f_{\text{root},0}(x).$$
}

Using the above bijection, we see that cover-inclusive Dyck
tilings are a natural generalization of the Dyck tableaux recently
introduced by Aval, Boussicault, and Dasse-Hartaut in \cite{abdh}.
They defined a \textbf{Dyck tableau} of order $n$ to be a skew shape
between the zig-zag path $\zigzag_{n+1}$ and upper Dyck path $\mu$
containing exactly $n$ dots, such that each column of the skew shape
going through the valleys of $\zigzag_{n+1}$ contain exactly one dot,
as shown in \fref{dyck-tableau}.

\begin{figure}[b]
\begin{center}
\tz{\includegraphics{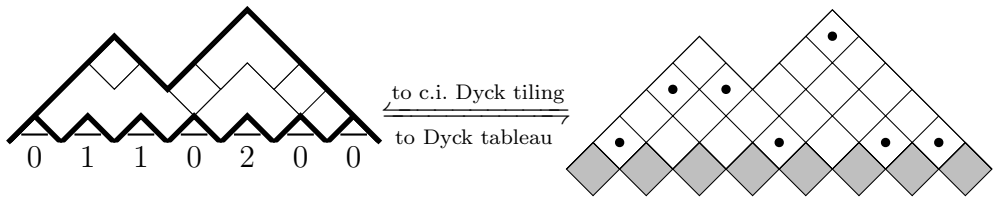}}{
\begin{tikzpicture}[rotate=45,scale=0.5]

\begin{scope}
\begin{scope}[shift={(-0.1,-0.1)}]
\draw[thick] (0.5,0)--(1,-0.5)  node[midway,shift={(0,-0.3)}] {0};
\draw[thick] (1.5,-1)--(2,-1.5)  node[midway,shift={(0,-0.3)}] {1};
\draw[thick] (2.5,-2)--(3,-2.5)  node[midway,shift={(0,-0.3)}] {1};
\draw[thick] (3.5,-3)--(4,-3.5)  node[midway,shift={(0,-0.3)}] {0};
\draw[thick] (4.5,-4)--(5,-4.5)  node[midway,shift={(0,-0.3)}] {2};
\draw[thick] (5.5,-5)--(6,-5.5)  node[midway,shift={(0,-0.3)}] {0};
\draw[thick] (6.5,-6)--(7,-6.5)  node[midway,shift={(0,-0.3)}] {0};
\end{scope}

\DyckPath[draw, line width =2pt]{4,-2,3,-5};
\DyckPath[clip] {4,-2,3,-5};
\draw (0,-7) grid (7,0);

\Ribbon[shift={(1,0)},fill=white]{2,-1,1,-2};
\Ribbon[shift={(4,-3)},fill=white]{2,-2};
\Ribbon[shift={(5,-2)},fill=white]{2,-2};
\DyckPath[draw, line width =2pt]{4,-2,3,-5};
\DyckPath[draw, line width =2pt]{1,-1,1,-1,1,-1,1,-1,1,-1,1,-1,1,-1};
\end{scope}

\begin{scope}[shift={(10,-11)}]

\draw[fill = lightgray] (0,0) rectangle +(1,-1);
\draw[fill = lightgray] (1,-1) rectangle +(1,-1);
\draw[fill = lightgray] (2,-2) rectangle +(1,-1);
\draw[fill = lightgray] (3,-3) rectangle +(1,-1);
\draw[fill = lightgray] (4,-4) rectangle +(1,-1);
\draw[fill = lightgray] (5,-5) rectangle +(1,-1);
\draw[fill = lightgray] (6,-6) rectangle +(1,-1);
\draw[fill = lightgray] (7,-7) rectangle +(1,-1);

\DyckPath[draw]{5,-2,3,-6};
\DyckPath[clip] {5,-2,3,-6};
\draw (0,-8) grid (8,0);

\DyckDot{1,0};
\DyckDot{3,0};
\DyckDot{4,-1};
\DyckDot{4,-3};
\DyckDot{7,-2};
\DyckDot{6,-5};
\DyckDot{7,-6};

\DyckPath[draw]{5,-2,3,-6};
\DyckPath[draw]{1,-1,1,-1,1,-1,1,-1,1,-1,1,-1,1,-1,1,-1};
\end{scope}
\node at (9.3,-8.3){$\xleftrightharpoons[\text{to Dyck tableau}]{\text{to c.i.\ Dyck tiling}}$};
\end{tikzpicture}
}
\end{center}
\caption{\label{dyck-tableau}
The bijection between cover-inclusive Dyck tilings with $\lambda=\zigzag_n$ and Dyck tableaux.  The dot heights encode the chord labels.}
\end{figure} 

\begin{proposition}
  \label{prop:tilings-tableaux}
  There is a bijection between the cover-inclusive Dyck tilings whose
  lower path is the zig-zag path $\zigzag_n=(\textrm{UD})^{n}$ of $n$
  up-down steps and Dyck tableaux of order $n$.
\end{proposition}
\newpage
\begin{proof}
  Let the dot-height of a dot in $\zigzag_{n+1}/\mu$ be the number of
  boxes in the column of the dot, which are in $\zigzag_{n+1}/\mu$ and
  below the dot.  These dot heights are naturally associated with the
  chords of $\zigzag_n$, and are independent of one another.  Now
  $P_{\zigzag_n}$ is just the antichain on $n$ points, so the weakly
  increasing condition is vacuously true, and the maximum dot heights
  are precisely the maximum number of Dyck tiles that can fit within
  $\lambda/\mu$ and cover the chord associated with the dot's column
  (see \fref{dyck-tableau}).  Thus, Dyck tilings with lower path
  $\zigzag_n$ and Dyck tableaux of order $n$ are different
  representations of the same object.
\end{proof}

With this interpretation of Dyck tableaux as cover-inclusive Dyck
tilings in $\D(\zigzag_n,*)$, the bijection $\DTr(\zigzag_n,\cdot)$ is
equivalent to the bijection given in \cite{abdh}.

\section{Dyck tilings and 231-avoiding permutations}
\label{tiling/231}

We consider Dyck tilings whose lower path $\lambda$ is the zigzag
path, i.e., $\lambda=\zigzag_n = (\text{UD})^n$.  The poset $P_\lambda$ is just an
antichain of $n$ points, so that its linear extensions are exactly the
permutations on $n$ letters, i.e., $\LL(P_{\zigzag_n}) = S_n$, the symmetric
group on $n$ letters.
We consider $231$-avoiding permutations in the usual sense ---
namely, permutations $\sigma$, such that there are no indices $i<j<k$
for which $\sigma_k<\sigma_i<\sigma_j$.

\begin{theorem}\label{231avdthm}
  The maps $\DTs(\zigzag_n,\cdot)$ and $\DTr(\zigzag_n,\cdot)$ restrict to
  bijections between 231-avoiding permutations in $S_n$ and Dyck
  tilings whose lower path is $\zigzag_n$ and which contain only one-box tiles.
\end{theorem}

\begin{proof}
  The map $\DTs$ places a tile above a position iff there is an
  element to the left of the position that is larger than an element
  to the right of the position, i.e., if it is surrounded by a
  ``$2,1$.''  Thus, the image of $\DTs$ has no long tiles iff there
  are no $231$'s in the permutation.  The map $\DTr$ places a tile
  above a position iff there is an element to the left of the position
  which is one larger than an element to the right of the position.
  But this occurs precisely when there is an element to the left of
  the position that is larger than an element to the right of the
  position.  Thus the image of $\DTr$ also has no long tiles iff there
  are no $231$'s in the permutation.
\end{proof}

Reflecting $\mu$ about the $y$-axis and reversing the permutation
$\sigma$, i.e., $\sigma_n,\ldots,\sigma_1$, gives a bijection between
Dyck paths and $132$-avoiding permutations.  For the case $\DTr$, this
bijection has been given by Knuth in
\cite[Problem~2.2.1--4]{knuth}. %pp 60-61

\section{Dyck tilings and the mad statistic}
\label{tiling/mad}

In this section we relate Dyck tilings to the permutation
statistic $\mad$, which was defined by Clarke, Steingr{\'{\i}}msson,
and Zeng \cite{CSZ}, and whose definition we now review.  For a word
$w=w_1\cdots w_n$ of order $n$, the descent set of $w$ is denoted by
$$\DES(w)=\{i<n :  w_i>w_{i+1}\}.$$
The statistic $\mad$ of a permutation $\sigma$ is then defined
by
\begin{align*}
\desdif(\sigma) &= \sum_{i\in\DES(\sigma)} (\sigma_i-\sigma_{i+1}),\\
\res(\sigma) &= \sum_{i\in\DES(\sigma)} \#\{k<i: \sigma_i>\sigma_k>\sigma_{i+1}\},\\
\mad(\sigma) &= \desdif(\sigma)+\res(\sigma).
\end{align*}
This can also be written as
$$
\mad(\sigma) = \sum_{i\in\DES(\sigma)} \big[1+\#\{k>i+1: \sigma_i>\sigma_k>\sigma_{i+1}\}+2\times\#\{k<i: \sigma_i>\sigma_k>\sigma_{i+1}\}\big].
$$
Clarke \textit{et al.} gave a bijective proof that $\mad$ is a Mahonian statistic \cite{CSZ}, i.e., that it is equidistributed with $\inv$.

In this section we prove
\begin{theorem}
  \label{mad/art}
  For each permutation $\sigma$ of order $n$,
  $$ \art(\DTr(\zigzag_n,\sigma)) = \mad(\sigma).$$
\end{theorem}
\noindent
When we combine \tref{mad/art} with \tref{inv/art}, we obtain an
involution $$\DTs(\zigzag_n,\cdot)^{-1} \circ \DTr(\zigzag_n,\cdot)$$
on permutations of order $n$ which goes by way of Dyck tilings and
which maps $\mad$ to $\inv$.  This involution of course shows that
$\mad$ is equidistributed with $\inv$.  We do not see any connection
between this involution and the one given by Clarke \textit{et al.}

\begin{lemma}
  Suppose $\lambda$ is a Dyck tiling and $\sigma\in\L$.  In the
  increasing planted plane tree associated with $\lambda$ and
  $\sigma$, let $\ell_i$ and $r_i$ be the left and right endpoints of
  the chord of $P_\lambda$ labeled $i$, as shown in \fref{extension}.
  Letting $T=\DTr(\lambda,\sigma)$, we have
\begin{align}
\area(T) &=  \sum_{i\in\DES(\ell)} (\ell_i - r_{i+1}), \label{area}\\
\tiles(T) &=  \sum_{i\in\DES(\ell)} (\ell_i - r_{i+1} -
2\times \#\{j>i+1: r_{i+1}<\ell_j<\ell_i \}),  \label{tiles}
\end{align}
where $\ell$ denotes the word $\ell_1\cdots\ell_n$.
\end{lemma}

\begin{proof}
  We let $\ell_i^{(k)}$ and $r_i^{(k)}$ denote the left and right
  endpoints of the chord labeled $i$ after the $k$\th $\rgrow$ map.
  Recall that $p_1,\dots,p_n$ gives the sequence of growth locations,
  and that $s_i=p_i-(i-1)$.  We have $\ell_i^{(i)}=p_i+1$ and
  $r_i^{(i)}=p_i+2$.  Suppose $i>1$.  If $s_i\geq s_{i-1}$, then the
  $i$\th $\rgrow$ map adds no new one-box tiles,
  $\ell_i^{(i)}>\ell_{i-1}^{(i)}$, and so
  $\ell_i^{(n)}>\ell_{i-1}^{(n)}$, so $i-1\notin\DES(\ell^{(n)})$.  If
  on the other hand $s_i<s_{i-1}$, then the number of one-box tiles
  that the $i$\th $\rgrow$ map adds is $s_{i-1}-s_i=p_{i-1}-p_i-1=
  \ell_{i-1}^{(i)}-r_i^{(i)}$, and $i-1\in\DES(\ell^{(n)})$.  Later on
  in the growth process, new chords may get added between $r_i$ and
  $\ell_{i-1}$, which of course increases the difference between them.
  For $j>i$, this happens iff
  $r_i^{(j)} < \ell_j^{(j)} < r_j^{(j)} < \ell_{i-1}^{(j)}$, which happens iff
  $r_i^{(j)} < \ell_j^{(j)} < \ell_{i-1}^{(j)}$, which happens iff
  $r_i^{(n)} < \ell_j^{(n)} < \ell_{i-1}^{(n)}$, and when this happens
  the distance between $r_i$ and $\ell_{i-1}$ increases by $2$ (i.e.,
  $r_i^{(j)}-\ell_{i-1}^{(j)} = r_i^{(j-1)}-\ell_{i-1}^{(j-1)} + 2$)
  and otherwise increases by $0$.  This establishes the tiles formula
  \eqref{tiles}.  Notice that any such chord $j$ also increases by $2$
  the area of one of the tiles produced by chord $i$.  This
  establishes the area formula \eqref{area}.
\end{proof}

For a word $w$, the inversion set of $w$ is denoted by
$$\INV(w)=\{(i,j):i<j\text{ and }w_j<w_i\}.$$

\begin{lemma}\label{prop:prepost}
  For a Dyck path $\lambda$ and a natural labeling $L$ of $P_\lambda$,
  let $\ell_i$ and $r_i$ denote the left and right endpoints of the
  chord labeled $i$.  Then $\INV(\ell) \subset \INV(r)$.
  In particular, if $\sigma$ and $\tau$ are the standardizations of
  $\ell$ and $r$ respectively, then
  $\INV(\sigma)\subset \INV(\tau)$.
\end{lemma}
\begin{proof}
  Suppose $i<j$.  Let $x$ and $y$ denote the chords of $P_\lambda$
  which are labeled $i$ and~$j$ by $L$.  Since $L$ is natural, either
  $x<y$ in $P_\lambda$, or $x$ and $y$ are incomparable in
  $P_\lambda$.

  If $x<y$ in $P_\lambda$, then $\ell_i<\ell_j<r_j<r_i$.  Thus
  $(i,j)$ is not an inversion of $\ell$, but it is an inversion of
  $r$.

  If $x$ and $y$ are incomparable in $P_\lambda$, then either
  $\ell_i < r_i < \ell_j < r_j$, or else $\ell_j< r_j < \ell_i < r_i$.
  In this case, $(i,j)$ is either an inversion in both
  $\ell$ and~$r$, or an inversion in neither of them.
\end{proof}

If $\sigma$ and $\tau$ are two permutations on $[n]$ with
$\INV(\sigma)\subset\INV(\tau)$, then we define
\begin{align*}
\desdif(\sigma,\tau) &= \sum_{i\in\DES(\sigma)} (\sigma_i-\sigma_{i+1} + \tau_i-\tau_{i+1}).
\end{align*}
For a word $w=w_1\cdots w_n$ and a descent $i\in\DES(w)$, we define the set $\REM_i(w)$
of \textbf{right embraced numbers of $w$ with respect to descent $i$} by
\[
\REM_i(w)=\{k>i:w_i>w_k>w_{i+1}\}.
\]

\begin{lemma} \label{mad(s,t)/art}
  Let $\lambda$ be a Dyck path, $\sigma\in\L$, and
  $T=\DTr(\lambda,\sigma)$.  Let $\ell_i$ and $r_i$ denote the left
  and right endpoints of the chord labeled $i$ in the increasing planted plane
  tree associated with $\lambda,\sigma$.  Recall that
  $\sigma$ is the standardization of $\ell$, and let $\tau$ denote the
  standardization of~$r$.
  Then
\begin{align}
\area(T) &= \desdif(\sigma,\tau)-\des(\sigma) -
  \sum_{i\in\DES(\sigma)} |\REM_i(\sigma)\Delta\REM_i(\tau)|,  \label{area-st}\\
\tiles(T)&=\desdif(\sigma,\tau) - \des(\sigma) -\sum_{i\in\DES(\sigma)} (|\REM_i(\sigma)|+|\REM_i(\tau)| ),  \label{tiles-st}
\end{align}
where $A\Delta B = (A\cup B)\setminus(A\cap B)$.
\end{lemma}
\begin{proof}
  Suppose $i\in\DES(\ell)$, i.e., $\ell_{i+1}<\ell_i$.  Then in fact
  $\ell_{i+1}<r_{i+1}<\ell_i<r_i$.  To characterize $\ell_i-r_{i+1}$,
  consider any other chord, say with label $j$ ($j\neq i$ and $j\neq
  i+1$), which has at least one step between $r_{i+1}$ and $\ell_i$.
  There are three cases as follows:
\begin{enumerate}
\item[Case 1:] Both $\ell_j$ and $r_j$ are between $r_{i+1}$ and $\ell_i$, i.e.,
  $\ell_{i+1}<r_{i+1}<\ell_j<r_j<\ell_i<r_i$.
  Because the chords are noncrossing, this case happens iff both
  $\ell_{i+1}<\ell_j<\ell_{i}$ and $r_{i+1}<r_j<r_{i}$.
\item[Case 2:] Only $\ell_j$ is between $r_{i+1}$ and $\ell_i$, which
  happens iff $\ell_{i+1}<r_{i+1}<\ell_j<\ell_i<r_i<r_j$.
  It is easy to see that this case occurs iff
  $j<i$ and $\ell_{i+1}<\ell_j<\ell_i$ and $r_{i+1}<r_i<r_j$.
\item[Case 3:] Only $r_j$ is between $r_{i+1}$ and $\ell_i$.
  This case is similar to case 2, and occurs iff $j<i$ and
  $\ell_j<\ell_{i+1}<\ell_i$ and $r_{i+1}<r_j<r_i$.
\end{enumerate}

By considering chords $j$ that fall into one of these three cases, and
using the fact that $\sigma$ is the standardization of $\ell$ and $\tau$
is the standardization of $r$, we obtain
\begin{align}
\ell_i-r_{i+1} =&
1\begin{aligned}[t]+ &2\times\\+&\\+&\end{aligned}
\begin{aligned}[t]
&\#\{j: \sigma_{i+1}<\sigma_j<\sigma_i\text{ and } \tau_{i+1}<\tau_j<\tau_i\}\\
&\#\{j<i: \sigma_{i+1}<\sigma_j<\sigma_i\text{ and } \tau_{i+1}<\tau_i<\tau_j\}\\
&\#\{j<i: \sigma_j<\sigma_{i+1}<\sigma_i\text{ and } \tau_{i+1}<\tau_j<\tau_i\}
\end{aligned} \notag\\
=&
1\begin{aligned}[t]+ &\\-&\\-&\end{aligned}
\begin{aligned}[t]
&(\sigma_{i}-\sigma_{i+1}-1) + (\tau_{i}-\tau_{i+1}-1)\\
&\#\{j>i: \sigma_{i+1}<\sigma_j<\sigma_i\text{ and } \tau_{i+1}<\tau_i<\tau_j\}\\
&\#\{j>i: \sigma_j<\sigma_{i+1}<\sigma_i\text{ and } \tau_{i+1}<\tau_j<\tau_i\}
\end{aligned} \notag\\
=& \sigma_{i}-\sigma_{i+1}+\tau_{i}-\tau_{i+1}-1  -
|\REM_i(\sigma)\Delta\REM_i(\tau)|. \label{liri+1}
\end{align}
Upon summing over descents of $\sigma$ and using \eqref{area}, we obtain
\eqref{area-st}.

Recall equation~\eqref{tiles}.
For $j>i+1$, we have $r_{i+1}<\ell_j<\ell_i$ if and only if
$r_{i+1}<r_j<\ell_i$, which in turn occurs if and only if both
$\ell_{i+1}<\ell_j<\ell_i$ and $r_{i+1}<r_j<r_i$.
Thus $$\#\{j>i+1: r_{i+1}<\ell_j<\ell_i \} = |\REM_i(\sigma)\cap\REM_i(\tau)|,$$
and combining this with \eqref{tiles} and \eqref{liri+1} yields \eqref{tiles-st}.
\end{proof}

\begin{proof}[Proof of \tref{mad/art}]
  We use \lref{mad(s,t)/art}, and observe that
  when $\lambda=\zigzag_n$ we have $\sigma=\tau$, so when
  we evaluate $\art(T)=(\area(T)+\tiles(T))/2$, we obtain
  $$ \art(\DTr(\zigzag_n,\sigma)) = \desdif(\sigma,\sigma)-\des(\sigma)
  -\sum_{i\in\DES(\sigma)} |\REM_i(\sigma)|.$$
Now $\desdif(\sigma,\sigma)=2\times\desdif(\sigma)$, and it is easy to see that
\[
\res(\sigma)=\desdif(\sigma)-\des(\sigma)-
\sum_{i\in\DES(\sigma)}\#\{k>i: \sigma_{i+1} < \sigma_k < \sigma_i\},
\]
so we obtain
\[
 \art(\DTr(\zigzag_n,\sigma)) = \desdif(\sigma)+\res(\sigma) = \mad(\sigma). \qedhere
\]
\end{proof}

\phantomsection
\pdfbookmark[1]{References}{bib}
\bibliographystyle{hmralpha}
\bibliography{dyck-tiling}

\def\cprime{$'$}
\begin{thebibliography}{ABDH13}

\bibitem[ABDH13]{abdh}
Jean-Christophe Aval, Adrien Boussicault, and Sandrine Dasse-Hartaut.
\newblock Dyck tableaux.
\newblock {\em Theoret.\ Comput.\ Sci.}, 502:195--209, 2013.
\newblock \arXiv{1109.0370}. \MR{3101701}

\bibitem[ABN11]{abn}
Jean-Christophe Aval, Adrien Boussicault, and Philippe Nadeau.
\newblock Tree-like tableaux.
\newblock In {\em 23rd {I}nternational {C}onference on {F}ormal {P}ower
  {S}eries and {A}lgebraic {C}ombinatorics ({FPSAC} 2011)}, Discrete Math.\
  Theor.\ Comput.\ Sci.\ Proc., pages 63--74. 2011.
\newblock \arXiv{1109.0371}. \MR{2820698 (2012m:05059)}

\bibitem[BFS92]{BFS}
Fran{\c{c}}ois Bergeron, Philippe Flajolet, and Bruno Salvy.
\newblock \href{http://hal.inria.fr/docs/00/07/49/77/PDF/RR-1583.pdf}{Varieties
  of increasing}
  \href{http://hal.inria.fr/docs/00/07/49/77/PDF/RR-1583.pdf}{trees}.
\newblock In {\em C{AAP} '92 ({R}ennes, 1992)}, Lecture Notes in Comput.\
  Sci.~\#581, pages 24--48. Springer, 1992. \MR{1251994 (94j:68233)}

\bibitem[BW89]{BW}
Anders Bj{\"o}rner and Michelle~L. Wachs.
\newblock {$q$}-{H}ook length formulas for forests.
\newblock {\em J. Combin.\ Theory Ser.\ A}, 52(2):165--187, 1989. \MR{1022316
  (91e:05013)}

\bibitem[CSZ97]{CSZ}
Robert~J. Clarke, Einar Steingr{\'{\i}}msson, and Jiang Zeng.
\newblock \href{http://www.mat.univie.ac.at/~slc/s/s35clarke.pdf}{New
  {E}uler-{M}ahonian}
  \href{http://www.mat.univie.ac.at/~slc/s/s35clarke.pdf}{statistics on
  permutations and words}.
\newblock {\em Adv.\ in Appl.\ Math.}, 18(3):237--270, 1997. \MR{1436481
  (97m:05008)}

\bibitem[Fay13]{fayers}
Matthew Fayers.
\newblock Dyck tilings and the homogeneous {G}arnir relations for graded
  {S}pecht modules.
\newblock 2013.
\newblock \arXiv{1309.6467}.

\bibitem[FN12]{fisher-nadeau}
Ilse Fischer and Philippe Nadeau.
\newblock Fully packed loops in a triangle: matchings, paths and puzzles.
\newblock 2012.
\newblock \arXiv{1209.1262}.

\bibitem[Kim12]{kim}
Jang~Soo Kim.
\newblock Proofs of two conjectures of {K}enyon and {W}ilson on {D}yck tilings.
\newblock {\em J. Combin.\ Theory Ser.~A}, 119(8):1692--1710, 2012.
\newblock \arXiv{1108.5558}. \MR{2946383}

\bibitem[Knu73]{knuth3}
Donald~E. Knuth.
\newblock {\em The Art of Computer Programming, {V}ol.~3: Sorting and
  searching}.
\newblock Addison-Wesley Publishing Co., 1973. \MR{0445948 (56 \#4281)}

\bibitem[Knu75]{knuth}
Donald~E. Knuth.
\newblock {\em The Art of Computer Programming, Vol.~1: Fundamental
  algorithms}.
\newblock Addison-Wesley Publishing Co., second edition, 1975. \MR{0378456 (51
  \#14624)}

\bibitem[KW11a]{kw}
Richard~W. Kenyon and David~B. Wilson.
\newblock Double-dimer pairings and skew {Y}oung diagrams.
\newblock {\em Electron.\ J. Combin.}, 18(1):Paper 130, 22, 2011.
\newblock \arXiv{1007.2006}. \MR{2811099 (2012g:05268)}

\bibitem[KW11b]{kw:annular}
Richard~W. Kenyon and David~B. Wilson.
\newblock Spanning trees of graphs on surfaces and the intensity of loop-erased
  random walk on {$\mathbb Z^2$}.
\newblock 2011.
\newblock \arXiv{1107.3377}.

\bibitem[LS81]{MR646823}
Alain Lascoux and Marcel-Paul Sch{\"u}tzenberger.
\newblock Polyn\^omes de {K}azhdan \& {L}usztig pour les grassmanniennes.
\newblock In {\em Young tableaux and {S}chur functors in algebra and geometry
  ({T}oru\'n, 1980)}, volume 87--88 of {\em Ast\'erisque}, pages 249--266.
  Soc.\ Math.\ France, 1981. \MR{646823 (83i:14045)}

\bibitem[Sta72]{MR0332509}
Richard~P. Stanley.
\newblock {\em Ordered Structures and Partitions}.
\newblock Memoirs of the AMS \#119. American Mathematical Society, 1972.
  \MR{0332509 (48 \#10836)}

\bibitem[Sta99]{stanley}
Richard~P. Stanley.
\newblock {\em \href{http://www-math.mit.edu/~rstan/ec/}{Enumerative
  Combinatorics}, {V}ol.~2}.
\newblock Cambridge studies in advanced mathematics \#62. Cambridge University
  Press, 1999.
\newblock With a foreword by G.-C. Rota and an appendix by S. Fomin.
  \MR{1676282 (2000k:05026)}

\bibitem[SZJ12]{MR2927185}
Keiichi Shigechi and Paul Zinn-Justin.
\newblock Path representation of maximal parabolic {K}azhdan-{L}usztig
  polynomials.
\newblock {\em J. Pure Appl.\ Algebra}, 216(11):2533--2548, 2012.
\newblock \arXiv{1001.1080}. \MR{2927185}

\end{thebibliography}

\end{document}